\documentclass[12pt]{article}
\usepackage{amsmath}
\usepackage{amssymb}
\usepackage{harpoon}
\usepackage{latexsym,color}
\usepackage[textures,dvipdfm]{hyperref}
\usepackage{ntheorem}
\usepackage{mathrsfs}
\usepackage[numbers,sort&compress]{natbib}

  \allowdisplaybreaks[4]
\title{  Classification of equilibria for the spatially  homogeneous Boltzmann equation  for  Fermi-Dirac particles
  }

\author{ Xuguang Lu
\footnote {Department of Mathematical Sciences, Tsinghua
University, Beijing 100084, People's Republic of China;  e-mail:
xglu@mail.tsinghua.edu.cn }}

\date{}      % Deleting this command produces today's date.

\begin{document}             % End of preamble and beginning of text.
\maketitle                   % Produces the title.

\newcommand{\ap}{\alpha}
\newcommand{\ld}{\lambda}
\newcommand{\td}{\tilde}
\newcommand{\ep}{\epsilon}
\newcommand{\p}{\partial}
\newcommand{\vp}{\varphi}
\newcommand{\vep}{\varepsilon}
\newcommand{\og}{\omega}
\newcommand{\Og}{\Omega}
\newcommand{\sg}{\sigma}
\newcommand{\Sg}{\Sigma}
\newcommand{\gm}{\gamma}
\newcommand{\Gm}{\Gamma}
\newcommand{\dt}{\delta}
\newcommand{\Dt}{\Delta}
\newcommand{\fr}{\frac}
\newcommand{\wt}{\widetilde}
\newcommand{\wh}{\widehat}
\newcommand{\intt}{\int\!\!\!\!\int}
\newcommand{\inttt}{\int\!\!\!\!\int\!\!\!\!\int}
\newcommand{\intttt}{\int\!\!\!\!\int\!\!\!\!\int\!\!\!\!\int}
\newcommand{\bR}{{\mathbb R}^3 }
\newcommand{\bS}{{\mathbb S}^2 }
\newcommand{\bSS}{{\bS}\times{\bS}}
\newcommand{\bSSS}{{\bS}\times{\bS}\times{\bS}}
\newcommand{\bRR}{{\bR}\times{\bR}}
\newcommand{\bRRR}{{\bRR}\times{\bR}}
\newcommand{\bRS}{{\bR}\times {\mathbb S}^2 }
\newcommand{\bRRS}{{\bRR}\times{\mathbb S}^2 }
\newcommand{\bRRRS}{{\bRRR}\times{\mathbb S}^2 }
\newcommand{\bRd}{{\mathbb R}^d}
\newcommand{\bRN}{{\mathbb R}^N}
\newcommand{\bRSN}{{\bRN}\times {\mathbb S}^{N-1} }
\newcommand{\bSN}{{\mathbb S}^{N-1}}
\newcommand{\bRRSN}{{\bRN}\times {\bRN}\times{\mathbb S}^{N-1} }
\newcommand{\bRn}{{\mathbb R}^n}
\newcommand{\bRSn}{{\bRn}\times {\mathbb S}^{n-1} }
\newcommand{\bSn}{{\mathbb S}^{n-1}}
\newcommand{\bRRn}{{\bRn}\times {\bRn}}
\newcommand{\bRRSn}{{\bRn}\times {\bRn}\times{\mathbb S}^{n-1} }
\newcommand{\la}{\langle}
\newcommand{\ra}{\rangle}
\newcommand{\mR}{{\mathbb R}}
\newcommand{\mN}{{\mathbb N}}
\newcommand{\mQ}{{\mathbb Q}}
\newcommand{\mC}{{\mathbb C}}
\newcommand{\mZ}{{\mathbb Z}}
\newcommand{\mS}{{\mathbb S}}
\newcommand{\be}{\begin{myequation}}
\newcommand{\ee}{\end{myequation}}
\newcommand{\bes}{\begin{myeqnarray}}
\newcommand{\ees}{\end{myeqnarray}}
\newcommand{\beas}{\begin{eqnarray*}}
\newcommand{\eeas}{\end{eqnarray*}}
\newcommand{\lb}{\label}
\newcounter{thm}
\setcounter{thm}{0}
\theoremseparator{.}
\newtheorem{theorem}{Theorem}[section]
\newtheorem{proposition}[theorem]{Proposition}
\newtheorem{definition}[theorem]{Definition}
\newtheorem{lemma}[theorem]{Lemma}
\newtheorem{remark}[theorem]{Remark}
\newtheorem{question}[theorem]{Question}
\newtheorem{problem}[theorem]{Problem}
\newtheorem{assumption}[theorem]{Assumption}
\newtheorem{corollary}[theorem]{\indent Corollary}
\newcounter{myequation}[section]
\renewcommand{\theequation}{\arabic{section}.\arabic{myequation}}
\newenvironment{proof}{{\bf Proof.}}{$\hfill\Box$}
\newenvironment{myequation}{\stepcounter{myequation}\begin{equation}}{\end{equation}}
\newenvironment{myeqnarray}{\stepcounter{myequation}\begin{eqnarray}}{\end{eqnarray}}
\newcommand{\dnumber}{\stepcounter{myequation}}

\vskip 0.1in \baselineskip 18.2pt

\begin{abstract}
The classification of equilibria for the spatially  homogeneous Boltzmann equation for Fermi-Dirac particles is proved for any $n$-dimensional velocity space with $n\ge 2$. The same classification
has been proven in \cite{Lu2001} for $n=3$. Now the proof for $n\ge 2$ is based on
a recent result on a characterization of Euclidean balls for all dimensions $\ge 2$.
 \\



\end{abstract}

\begin{center}\section{ Introduction } \end{center}

The spatially  homogeneous Boltzmann equation  for  Fermi-Dirac particles
under consideration is given by (after normalizing parameters)
$$
\fr{\p }{\p t}f(v,t) =\intt_{\bRSn}B(v-v_*,\sg)
\big(f'f_*'(1-f)(1-f_*)-ff_*(1-f')(1-f_*')
\big){\rm d}v_*{\rm d}\sg \eqno({\rm BFD})$$
with $(v, t)\in{\mathbb R}^n\times(0,\infty)$. Throughout this paper we assume that
$n\ge 2$.  
For $n=3$, such Boltzmann equations for Bose-Einstein particles and for Fermi-Dirac particles were first derived by Nordheim \cite{Nordheim} and Uehling $\&$ Uhlenbeck \cite{Uehling and Uhlenbeck}
 and then taken attention and developed by \cite{weak-coupling},\cite{CC},\cite{ESY},\cite{LS}.
The solution $f=f(v,t)\ge 0$ is the  number density of particles at time $t$ with the velocity $v$,
and as usual we denote briefly
$f_*=f(v_*,t), f'=f(v',t),f_*'=f(v_*',t)$  where $v,v_*$ and $v',v_*'$ are velocities of two particles just before and after
their collision:
\be v'=\fr{v+v_*}{2}+ \fr{|v-v_*|\sg}{2},\quad v_*'=\fr{v+v_*}{2}-\fr{|v-v_*|\sg}{2}, \qquad \sg\in{\bSn}
\lb{1.1}\ee
which conserves the momentum and kinetic energy
$$ v'+v_*'=v+v_*,\quad |v'|^2+|v_*'|^2=|v|^2+|v_*|^2.$$
The right hand side of Eq.(BFD) is called collision integral
which describes
the rate of change of $f$ due to the binary collision (\ref{1.1}). The function
$B(v-v_*,\sg)$ (the collision kernel) is a positive Borel function of
$(|v-v_*|, \la v-v_*,\sg\ra/|v-v_*|)$ only, and satisfies some conditions
so that Eq.(BFD) has solutions.

A solution $f$ of Eq.(BFD) is not only required having finite mass and energy, i.e. for every $t\in [0,\infty)$
 $$f(\cdot,t)\in L^1_2({\bRn})=\{ f\in L^1({\bRn})\,\,|\,\, \|f\|_{L^1_2}:=\int_{{\bRn}}(1+|v|^2)|f(v)|{\rm d}v<\infty\}$$
 but also satisfying the $L^{\infty}$-condition (due to the Pauli's exclusion principle):
\be 0\le f(v,t)\le 1\qquad \forall\, (v,t)\in {\bRn}\times [0,\infty).\lb{1.2}\ee
It is well-known that under some conditions on $B(v-v_*,\sg)$,  
if an initial datum  $f|_{t=0}=f_0\in L^1_2({\bRn})$ satisfies 
$0\le f_0\le 1$ on ${\bRn}$, then the corresponding 
solution $f$ satisfies (\ref{1.2}).

The entropy functional for Eq.(BFD) is
$$ S(f)=\int_{{\bRn}}\big\{-(1-f(v))
\log(1-f(v))- f(v)\log(f(v))\big\}{\rm d}v$$
from which it is easily seen that
$$0\le S(f)<\infty\quad {\rm for\,\,\, all}\,\,\,f\in L^1({\bRn})\,\,\, {\rm satisfying}\,\,\,
0\le f\le 1\quad {\rm on}\quad {\bRn}.$$
The corresponding entropy identity for solutions of Eq.(BFD) is
\be S(f(t))=S(f_0)+
\int_{0}^{t}D(f(s)){\rm d}s, \ \ \ t\ge 0 \lb{1.4}\ee
where  $f(t)=f(\cdot, t)$ and
$$D(f)=\fr{1}{4}\inttt_{\bRRSn}B(v-v_*,\sg)
\Gm\big(f'f'_*(1-f)(1-f_*)\,,\,ff_*(1-f')
(1-f_*')\big)
{\rm d}v_*{\rm d}\sg{\rm d}v,$$
$$  \Gm(a,b)=\left\{ \begin{array}{lll}(a-b)\log(a/b) \,,
& \mbox{  $a>0,\ b>0;$} \\
\ \ \ \infty  , & \mbox{  $a>0,\ b=0$ \ \ or
\ \ $a=0,\ b>0;$} \\
\ \ \ \ 0 ,  & \ \mbox{  $a=b=0.$} \end{array} \right.  $$

An equilibrium $f\in L^1_2({\bRn})$ of Eq.(BFD) is a
time-independent solution of the equation.
By entropy identity (\ref{1.4}) and $B(\cdot,\cdot)>0$ a.e.,
this is equivalent to saying that an equilibrium
$f(v)$ of Eq.(BFD) is  a
solution  of the  following functional equation
\be f'f_*'(1-f)(1- f_*)=f\,f_*(1-f')
(1-f_*') \ \ \ {\rm a.e.\,\,\,\, on}\quad {\bRRSn}\lb{1.4}\ee
together with  the  physical  conditions 
\be f\in L^1_2({\bRn}),\,\,\,\, \|f\|_{L^1_2}\neq 0\,\,\,\, {\rm and}\,\,\,\, 
0\leq f\leq 1\,\,\,\,{\rm on}\,\,\,\,{\bRn}.\lb{1.4*}\ee

Let $mes(E)$ be the Lebesgue measure of a Lebesgue measurable set $E\subset {\bRn}$. 
Let $|{\mS}^{n-1}|$ denote the area of the unit sphere ${\mS}^{n-1}\subset {\bRn}$.
As usual we denote ${\bRn}(0<f<1)=\{v\in {\bRn}\,|\, 0<f(v)<1\}, 
{\bRn}(f=1)=\{v\in{\bRn}\,|\, f(v)=1\}$, etc. 
The main result of the paper is the following theorem which gives a classification 
of solutions of Eq.(\ref{1.4}) with (\ref{1.4*}).
\vskip2mm

\begin{theorem}\label{theorem1.1}
 Let $n\ge 2$. Let $f$ be a solution of Eq.(\ref{1.4}) with (\ref{1.4*}) and let
$$M_0=\int_{\bRn}f(v){\rm d}v,\ \ \ \ M_2=\int_{\bRn}f(v)|v-v_0|^2{\rm d}v, \ \ \ \
v_0=\fr{1}{M_0}\int_{\bRn}f(v)v{\rm d}v.$$
Then $f$ is classified into (I) and (II) that are characterized as follows:

(I) $f$ satisfies one of the following three conditions that are equivalent to each other:
\bes&& S(f)>0\quad {\rm i.e.}\quad mes({\bRn}(0<f<1))>0.\lb{1.9}\\
&&
\fr{M_2}{(M_0)^{(n+2)/n}}> \fr{n}{n+2}\bigg(\fr{n}{|{\bSn}|}\bigg)^{2/n}\,.\dnumber \lb{1.10}\\
&&
f(v)=\fr{ae^{-b|v-v_0|^2}}{1+ae^{-b|v-v_0|^2}}\qquad {\rm a.e.}\quad v\in{\bRn} \dnumber\lb{1.11}\ees
for some constants $0<a, b<\infty$.

(II) $f$ satisfies one of the following three conditions that are equivalent to each other:
\bes&& S(f)=0\quad {\rm i.e.}\quad mes({\bRn}(0<f<1))=0.\lb{1.12}\\
&&\fr{M_2}{(M_0)^{(n+2)/n}}= \fr{n}{n+2}\bigg(\fr{n}{|{\bSn}|}\bigg)^{2/n}\,.
\dnumber \lb{1.13}\\
&& f(v)={\bf 1}_{\{|v-v_0|\le R\}}\qquad {\rm a.e.}\quad v\in{\bRn}\dnumber \lb{1.14}\ees
for some constant $0<R<\infty$.
\end{theorem}

For $n=3$, the classification given in Theorem \ref{theorem1.1} has been proven in \cite{Lu2001} where the inequality (\ref{1.10}) and the equality
(\ref{1.13}) are written in terms of the kinetic temperature $T>0$ and the critical kinetic temperature
$T_F>0$ (the Fermi-temperature) as follows:
$$T/T_F> 2/5\qquad {\rm and}\qquad T/T_F= 2/5.$$
Theorem \ref{theorem1.1} now shows that the same classification holds also true for all $n\ge 2$. In particular it holds true for $n=2$, which was once the most difficult case. The method of proof of Theorem \ref{theorem1.1} is similar to that of \cite{Lu2001} where
the key step is to deduce a characterization of Euclidean balls for $n\ge 3$.
Now since a recent work (see Appendix) proves that the same characterization of Euclidean balls
holds also true for all $n\ge 2$, it follows that Theorem \ref{theorem1.1} holds true. For completeness we give here a proof of Theorem \ref{theorem1.1}. To do this we first collect and prove some lemmas in Section 2, and then in Section 3 we prove Theorem \ref{theorem1.1}. The characterization of Euclidean balls mentioned above
and its proof are put in the Appendix.

\begin{center}\section{ Some Lemmas} \end{center}

\begin{lemma}\label{lemma2.1}(\cite{Lu2001})  Given constants $0<p<q<\infty$. Let
$\phi$ be measurable on $[0, \infty)$
with $0\le \phi\le 1$ on $[0,\infty)$ and $0<\int_{0}^{\infty}r^{q-1}\phi(r){\rm d}r<\infty.$
Then
$$\bigg(p\int_{0}^{\infty}r^{p-1}\phi(r){\rm d}r\bigg)^{1/p}
\le\bigg(q\int_{0}^{\infty}r^{q-1}\phi(r){\rm d}r\bigg)^{1/q}$$
and the equality holds if and only if 
there is a  constant $0<R<\infty$ such  that $
\phi={\bf 1}_{[0, R]}$  a.e. on $[0,\infty).$
\end{lemma}

\begin{lemma}\label{lemma2.2} Let  $f\in L^1_2({\bRn})$  satisfy $\|f\|_{L^1_2}\neq 0$ and
$ 0\le f\le 1 $ on ${\bRn}$, and let
$$M_0=\int_{\bRn}f(v){\rm d}v,\ \ \ \ M_2=\int_{\bRn}f(v)|v-v_0|^2{\rm d}v, \ \ \ \
v_0=\fr{1}{M_0}\int_{\bRn}f(v)v{\rm d}v.$$
Then
\be \fr{M_2}{(M_0)^{(n+2)/n}}\ge \fr{n}{n+2}\bigg(\fr{n}{|{\bSn}|}\bigg)^{2/n}
\lb{2.1}\ee
and the  equality  holds  if  and  only  if  $f(v)={\bf 1}_{\{|v-v_0|\le R\}}$ for some
constant $0<R<\infty$. 
\end{lemma}

\begin{proof} Let
$${\bar f}(r)=\fr{1}{|{\bSn}|}\int_{\bSn}f(v_0+r\sg){\rm d}\sg.$$
Then $0\le \bar{f}(r)\le 1$ for all $r\in [0,\infty)$ and
(\ref{2.1}) is equivalent to the inequality
$$\bigg((n+2)\int_{0}^{\infty}r^{n+1}{\bar f}(r){\rm d}r\bigg)^{1/(n+2)}
\ge \bigg(n\int_{0}^{\infty}r^{n-1}{\bar f}(r){\rm d}r\bigg)^{1/n} $$
which does hold by Lemma \ref{lemma2.1}. Also,
since $0\le f\le 1$ on ${\bRn}$, it is easily seen that the two equalities
${\bar f}(r)={\bf 1}_{[0, R]}(r)$  a.e. $ r\in[0,\infty)$
and
$f(v)={\bf 1}_{\{|v-v_0|\le R\}}$  a.e. $v\in{\bRn}$  are
equivalent. This proves the lemma.
\end{proof}
\vskip2mm

\begin{lemma}\label{lemma2.3}  Let $M_0>0, M_2>0,$  and  $v_0\in{\bRn}.$
Then  there  exists  a unique  Fermi-Dirac distribution
${F}_{a,b}(v)=\fr{a e^{-b|v-v_0|^2}}{1+a e^{-b|v-v_0|^2}}$
with the coefficients
$a>0, b>0$ such  that
\be\int_{\bRn}F_{a,b}(v){\rm d}v=M_0,\quad
\int_{\bRn}F_{a,b}(v)|v-v_0|^2{\rm d}v=M_2 \lb{2.2}\ee
if  and  only  if
$M_0,  M_2$  satisfy the strict inequality 
\be \fr{M_2}{(M_0)^{(n+2)/n}}>\fr{n}{n+2}\bigg(\fr{n}{|{\bSn}|}\bigg)^{2/n}.\lb{2.3}\ee
\end{lemma}

\begin{proof} Let
$$I_{s}(t)=\int_{0}^{\infty}\fr{r^{s}}{1+t e^{r^2}}
{\rm d}r,\quad t>0,\quad s\ge 0.$$
For any $k\ge 0$ we compute
$$\int_{\bRn}F_{a,b}(v)|v-v_0|^k{\rm d}v
= |{\bSn}|\big(\fr{1}{\sqrt{b}}\big)^{n+k}I_{n-1+k}(1/a).$$
Then (\ref{2.2}) is rewritten
\be |{\bSn}|\big(\fr{1}{\sqrt{b}}\big)^nI_{n-1}(1/a)=M_0,\quad
|{\bSn}|\big(\fr{1}{\sqrt{b}}\big)^{n+2}I_{n+1}(1/a)=M_2.\lb{2.4}\ee
Let
$$P(t)=\fr{I_{n+1}(t)}{\big(I_{n-1}(t)\big)^{\fr{n+2}{n}}},\quad t>0.$$
Then (\ref{2.4}) is equivalent to
\be \fr{1}{|{\bSn}|^{2/n}}P(1/a)=\fr{M_2}{(M_0)^{(n+2)/n}},\quad
b=\Big(\fr{I_{n-1}(1/a)|{\bSn}|}{M_0}\Big)^{2/n}.\lb{2.5}\ee
We now prove that
\be\fr{{\rm d}}{{\rm d}t}P(t)>0\quad \forall\, t>0;\quad 
\lim\limits_{t\to 0+}P(t)=\fr{n^{(n+2)/n}}{n+2},\quad 
\lim\limits_{t\rightarrow \infty}P(t)=\infty.\lb{2.6}\ee
Compute
\beas&& \fr{{\rm d}}{{\rm d}t}P(t)
=\big(\fr{{\rm d}}{{\rm d}t}I_{n+1}(t)\big)\big(I_{n-1}(t)\big)^{-\fr{n+2}{n}}
+I_{n+1}(t)(-\fr{n+2}{n})\big(I_{n-1}(t)\big)^{-\fr{n+2}{n}-1}
\fr{{\rm d}}{{\rm d}t}I_{n-1}(t),\\
&&
-\fr{{\rm d}}{{\rm d}t}I_s(t)=J_s(t):=\int_{0}^{\infty}\fr{r^{s} e^{r^2}}{(1+t e^{r^2})^2}
{\rm d}r,\quad t>0,\quad s\ge 0.\eeas
And integration by parts gives
\beas I_s(t)
=\fr{2t}{s+1}J_{s+2}(t),\quad t>0, \quad s\ge 0. \eeas
So
\beas\fr{{\rm d}}{{\rm d}t}P(t)
=\big(\fr{2t}{n}\big)^{-\fr{n+2}{n}}\big(J_{n+1}(t)\big)^{-\fr{n+2}{n}-1}
\big( J_{n-1}(t)J_{n+3}(t)-[J_{n+1}(t)]^2\big),\quad t>0.\eeas
Since by Cauchy-Schwarz inequality
$$[J_{n+1}(t)]^2
<J_{n-1}(t)J_{n+3}(t)\quad \forall\, t>0$$
it follows that $\fr{{\rm d}}{{\rm d}t}P(t)>0$ for all $t>0$.

In order to prove the first limit in (\ref{2.6}), we let
$t= e^{-\rho}, \rho>0$, and compute
\beas&& I_{s}(e^{-\rho})=\fr{\rho^{\fr{s+1}{2}}}{2}\int_{0}^{\infty}\fr{u^{\fr{s-1}{2}}}{1+ e^{\rho(u-1)}}
{\rm d}u,\quad s\ge 0,\\
&&P(e^{-\rho})=\fr{I_{n+1}(e^{-\rho})}{\big(I_{n-1}(e^{-\rho})\big)^{\fr{n+2}{n}}}
=2^{\fr{2}{n}}\fr{K_{n}(\rho)}{\big(K_{n-2}(\rho)\big)^{\fr{n+2}{n}}}
\eeas
where
$$K_s(\rho)=\int_{0}^{\infty}\fr{u^{\fr{s}{2}}}{1+ e^{\rho(u-1)}}{\rm d}u,\quad s\ge 0.$$
Using Lebesgue dominated convergence we have
\beas
K_s(\rho)=\int_{[0,1)}\fr{u^{\fr{s}{2}}}{1+ e^{\rho(u-1)}}{\rm d}u
+\int_{(1,\infty)}\fr{u^{\fr{s}{2}}}{1+ e^{\rho(u-1)}}{\rm d}u
\to \int_{[0,1)}u^{\fr{s}{2}} {\rm d}u=\fr{2}{s+2}\quad (\rho\to \infty).\eeas
So
\beas\lim\limits_{t\to 0+}P(t)=\lim_{\rho\to\infty}P(e^{-\rho})
=2^{\fr{2}{n}}\fr{\fr{2}{n+2}}{\big(\fr{2}{n}\big)^{\fr{n+2}{n}}}
=\fr{n^{\fr{n+2}{n}}}{n+2}
\eeas
i.e.
$$\fr{1}{|{\bSn}|^{2/n}}P(0+)=\lim_{\rho\to\infty}\fr{1}{|{\bSn}|^{2/n}}P(e^{-\rho})=\fr{n}{n+2}\bigg(\fr{n}{|{\bSn}|}\bigg)^{2/n}.$$
And for all $t>1$ we have
\beas P(t)
\ge t^{\fr{2}{n}}\fr{\int_{0}^{\infty}\fr{r^{n-1}}{1+e^{r^2}}
{\rm d}r}{\big(\int_{0}^{\infty}\fr{r^{n-1}}{e^{r^2}}
{\rm d}r\big)^{\fr{n+2}{n}}}\to \infty\quad (t\to\infty).
\eeas
So
$P(\infty)=\infty$.

Now suppose that (\ref{2.2}) or equivalently (\ref{2.5}) has a solution $a>0,b>0$. Then
$$\fr{M_2}{(M_0)^{(n+2)/n}}=\fr{1}{|{\bSn}|^{2/n}}P(1/a)
> \fr{1}{|{\bSn}|^{2/n}}P(0+)=\fr{n}{n+2}\bigg(\fr{n}{|{\bSn}|}\bigg)^{2/n}.$$
Conversely if $M_0,M_2$ satisfy the inequality (\ref{2.3}), then
$$\fr{1}{|{\bSn}|^{2/n}}P(0+)<\fr{M_2}{(M_0)^{(n+2)/n}}<\infty=
\fr{1}{|{\bSn}|^{2/n}}P(\infty)$$
and so there is a unique $a>0$ such that
$$\fr{1}{|{\bSn}|^{2/n}}P(1/a)=\fr{M_2}{(M_0)^{(n+2)/n}}.$$
With this unique $a>0$, the equation (\ref{2.5}), or equivalently, the equation
(\ref{2.2}), has a unique solution $a>0, b>0$.
\end{proof}
\vskip2mm

{\bf Remark.} If $f$ in Lemma \ref{lemma2.2}
is an indicator function $f(v)={\bf 1}_{E}(v)$
of a measurable set $E\subset {\bRn}$ which is essentially not a ball, i.e.
$mes\big((E\setminus B)\cup(B\setminus E)\big)>0$ for every ball $B\subset {\bRn}$, then according
to Lemma \ref{lemma2.2} the strict inequality (\ref{2.3}) holds for this $f={\bf 1}_{E}$.
This shows that  Lemma \ref{lemma2.2} and Lemma \ref{lemma2.3} are far from enough
to prove the classification of solutions of Eq.(\ref{1.4}) with (\ref{1.4*}).

\begin{lemma}\label{lemma2.4}  Let $b(t), \Psi(r), f(v)$ be nonnegative Borel functions on $[0,1], [0,\infty)$ and
${\bRn}$ respectively, and let ${\bf n}=(v-v_*)/|v-v_*|$. Then
\beas&&\intt_{{\bRSn}}b(|\la {\bf n},\sg\ra|)\Psi(|v-v_*|)f(v'){\rm d}v_*{\rm d}\sg\\
&&=
\intt_{{\bRSn}}b(|\la {\bf n},\sg\ra|)\Psi(|v-v_*|)f(v_*'){\rm d}v_*{\rm d}\sg\\
&&=2^{n-1}|{\mS}^{n-2}|
\int_{0}^{\pi/2}\fr{\sin^{n-2}(\theta)}{\cos^{2}(\theta)}b(|\cos(2\theta)|)
\Big(\int_{{\bRn}}\Psi\big(\fr{|v-v_*|}{\cos(\theta)}\big)f(v_*){\rm d}v_*\Big){\rm d}\theta. \eeas
Here for the case $n=2$, we define $|{\mS}^0|=2$.
\end{lemma}

\begin{proof} The first equality follows from the definition of $v', v_*'$.
The second equality is a special case of Lemma \ref{lemma5.4} in the Appendix by taking 
(for every fixed $v\in {\bRn}$)  $F(\la{{\bf n},\sg}\ra, v_*', v_*)=
b(|\la{{\bf n},\sg\ra}|)\Psi(|v-v_*|)f(v_*')$ and noting that
if $\tilde{\sg}\in {\mS}^{n-1}$ satisfies $\tilde {\sg}\bot {\bf n}$,
then $|v-v_*-|v-v_*|\tan(\theta)\tilde{\sg}|\\=|v-v_*||{\bf n}-\tan(\theta)\tilde{\sg}|
=|v-v_*|/\cos(\theta)$.
\end{proof}
\vskip2mm

{\bf Note.} Since the Lebesgue measure is regular, when dealing with 
measurability problems for functions in $L^1({\bRn})$,
we may automatically assume that they are all Borel functions on ${\bRn}$.
In other words, for notational
convenience we will do not distinguish between two measurable functions that differ by a null set.

\begin{lemma}\label{lemma2.5} Let $b(t)=1-t^2, t\in [0,1]; {\bf n}=(v-v_*)/|v-v_*|$.
For any function $f\in L^1({\bRn})$ satisfying $0\le f\le 1$ on ${\bRn}$, define
\begin{eqnarray*}
&& {\cal I}_f(v)=\intt_{\bRSn}b(|\la {\bf n},\sg\ra|)f(v')
f(v_*')(1-f(v_*)){\rm d}v_*{\rm d}\sg,\quad v\in{\bRn}\\
&& {\cal J}_f(v)=\intt_{\bRSn}b(|\la {\bf n},\sg\ra|)f(v_*)
(1-f(v'))(1-f(v_*')){\rm d}v_*{\rm d}\sg,\quad v\in{\bRn}.
\end{eqnarray*}
Then ${\cal I}_f,{\cal J}_f$ are well-defined and have the following properties:

(a). If $g\in L^1({\bRn})$ satisfying $0\le g\le 1$ on ${\bRn}$, then
$$|{\cal I}_f(v)-{\cal I}_g(v)|,\ \  |{\cal J}_f(v)-{\cal J}_g(v)|
\le C\|f-g\|_{L^1}\quad \forall\, v\in {\bRn}$$
where $C=(2^{n+2}+2)|{\mS}^{n-2}|\int_{0}^{\pi/2}\sin^n(\theta){\rm d}\theta$.
In particular, if $f=g$ a.e. on ${\bRn}$, then ${\cal I}_f\equiv {\cal I}_g$,
${\cal J}_f\equiv {\cal J}_g$
on ${\bRn}$. 

(b). ${\cal I}_f$, ${\cal J}_f$ are continuous on ${\bRn}$.

(c). If $mes({\bRn}(0<f<1))>0$,  then ${\bRn}({\cal I}_f>0)\cap{\bRn}({\cal J}_f>0)$ is non-empty.
\end{lemma}

\begin{proof} The proof of that ${\cal I}_f,{\cal J}_f$ are well-defined on ${\bRn}$ 
is included in the proof of property (a): choose $g=0$ in property (a) we have 
${\cal I}_g=0$, ${\cal J}_g=0$ on ${\bRn}$ and so $0\le {\cal I}_f(v), 
{\cal J}_f(v)\le C\|f\|_{L^1}$ for all $v\in {\bRn}$.

Proof of (a): Using Lemma \ref{lemma2.4} and $b(|\cos(2\theta)|)=\sin^2(2\theta)$ we have
\begin{eqnarray*}
&& |{\cal I}_f(v)-{\cal I}_g(v)|,\ \  |{\cal J}_f(v)-{\cal J}_g(v)|\\
&& \le \intt_{\bRSn}b(|\la {\bf n},\sg\ra|)\big[|(f-g)(v')|+|(f-g)(v_*')|+
|(f-g)(v_*)|\big]
{\rm d}v_*{\rm d}\sg \\
&& \le C\|f-g\|_{L^1} \qquad \forall \ v\in{\bRn}.
\end{eqnarray*}

Proof of (b): Denote
$f_h(v)=f(v+h)$.  We have
${\cal I}_f(v+h)={\cal I}_{f_h}(v), {\cal J}_f(v+h)={\cal J}_{f_h}(v)$ and it follows from 
property (a) that
$$|{\cal I}_f(v+h)-{\cal I}_f(v)|,\ \  |{\cal J}_f(v+h)-{\cal J}_f(v)|\le C\|f_h-f\|_{L^1}\quad
\forall\, v,h\in{\bRn}.$$
So ${\cal I}_f$, ${\cal J}_f$ are continuous on ${\bRn}$.

Proof of (c): Suppose $mes({\bRn}(0<f<1))>0$.
Then there is a Lebesgue point $v\in {\bRn}$ of $f$ satisfying
$0<f(v)<1$. Let $B_r(v)$ denote a closed ball in ${\bRn}$ with center $v$ and radius
$r>0$, and let 
$$L_{v}(r)=\fr{1}{mes(B_r)}\int_{B_r(v)}|f(v_*)-f(v)|{\rm d}v_*, \quad r>0.$$
Then $r\mapsto L_{v}(r)$ is bounded and $L_{v}(r)\to 0$ as $r\to 0+$ because 
$v$ is a Lebesgue point of $f$. Let
\beas A=\int_{{\bSn}}
b(|\la {\bf n},\sg\ra|){\rm d}\sg
=|{\mS}^{n-2}|\int_{0}^{\pi}
b(|\cos(\theta)|)\sin^{n-2}(\theta){\rm d}\theta. \eeas
In Lemma \ref{lemma2.4}, let us choose $\Psi(r)={\bf 1}_{\{0\le r\le\dt\}}$ for $\dt>0$. 
Using $0\le f\le 1$ on ${\bRn}$ gives
$$|f'f_*'(1-f_*)-f^2(1-f)|\le |f'-f|+|f_*'-f|+|f_*-f|.$$
From these we have
\begin{eqnarray*}
&& \Big|\fr{1}{mes(B_{\dt})}\intt_{{B_{\dt}(v)}\times{\bSn}}
b(|\la {\bf n},\sg\ra|)f'f_*'(1-f_*) {\rm d}v_*{\rm d}\sg-A[f(v)]^2(1-f(v))\Big|\\
&& \le\fr{1}{mes(B_{\dt})}\intt_{\bRSn}b(|\la {\bf n},\sg\ra|)
{\bf 1}_{\{|v-v_*|\le\dt\}}\\
&& \qquad \qquad \qquad \times\big(|f(v')-f(v)|+|f(v_*')-f(v)|+|f(v_*)-f(v)|\big)
{\rm d}v_*{\rm d}\sg\\
&& =2\fr{1}{mes(B_{\dt})}\intt_{\bRSn}b(|\la {\bf n},\sg\ra|)
{\bf 1}_{\{|v-v_*|\le\dt\}}|f(v')-f(v)|{\rm d}v_*{\rm d}\sg+AL_{v}(\dt)
\\
&&=2^{n}|{\mS}^{n-2}|
\int_{0}^{\pi/2}\fr{\sin^{n-2}(\theta)}{\cos^{2}(\theta)}b(|\cos(2\theta)|)
\Big(\int_{{\bRn}}{\bf 1}_{\{\fr{|v-v_*|}{\cos(\theta)}\le\dt\}}|f(v_*)-f(v)|{\rm d}v_*\Big){\rm d}\theta
+AL_{v}(\dt)
\\
&&=4|{\mS}^{n-2}|
\int_{0}^{\pi/2}\sin^{n}(2\theta)L_{v}(\dt\cos(\theta)){\rm d}\theta
+AL_{v}(\dt)\rightarrow 0 \ \ \
(\dt\rightarrow 0)
\end{eqnarray*}
since $L_{v}(r)$ is bounded and $L_{v}(r)\to 0 \ ( r\to 0+).$  Thus for sufficiently small $\dt>0$,
$${\cal I}_f(v)\ge\intt_{{B_{\dt}(v)}\times{\bS}}
b(|\la {\bf n},\sg\ra|)f'f_*'(1-f_*){\rm d}v_*{\rm d}\sg> \fr{1}{2}mes(B_{\dt})A[f(v)]^2(1-f(v)) >0.$$
Similarly, using $f(v)[1-f(v)]^2$ we also have
${\cal J}_f(v)>0$. Thus $v\in {\bRn}({\cal I}_f>0)\cap{\bRn}({\cal J}_f>0)$.
\end{proof}

\begin{center}\section{ Proof of  Theorem \ref{theorem1.1} } \end{center}

Let $f$ be given in Theorem \ref{theorem1.1}. 
Then we have, for instance, either $S(f)>0$ or $S(f)=0$. Thus in the following
we need only to prove that (\ref{1.9}),(\ref{1.10}),(\ref{1.11}) in part (I) are equivalent to each other, and
(\ref{1.12}),(\ref{1.13}),(\ref{1.14}) in part (II) are equivalent to each other.

Part (I):  Proof of ``$(\ref{1.9})\Longrightarrow(\ref{1.11})"$:
Suppose $S(f)>0$. By definition of $S(f)$, this is equivalent to
$mes({\bRn}(0<f<1))>0.$
We now  prove that in this case $f$ is a Fermi-Dirac distribution.
Let ${\cal I}_f(v), {\cal J}_f(v)$ be defined in Lemma \ref{lemma2.5}.
By Lemma \ref{lemma2.5}, ${\cal I}_f$, ${\cal J}_f$ are continuous on ${\bRn}$
and the set ${\bRn}({\cal I}_f>0)\cap{\bRn}({\cal J}_f>0)$ is non-empty.
Recall that $f$ satisfies the equation (\ref{1.4}).
Multiplying both sides of equation (\ref{1.4}) by $b(|\la {\bf n},\sg\ra|)$ and then taking
integration with respect to $(v_*,\sg)$
we deduce, for a null set $Z\subset {\bRn}$, that
\be f(v)\left [{\cal I}_f(v)+{\cal J}_f(v)\right ]={\cal I}_f(v)\qquad \forall\,v\in{\bRn}\setminus Z.
\lb{3.1}\ee
Define 
$$g(v)=\left\{ \begin{array}{lll} {\cal I}_f(v)/\left [{\cal I}_f(v)+{\cal J}_f(v)\right ]\quad
& \mbox{if\,\, ${\cal I}_f(v)+{\cal J}_f(v)>0$} \\ \\
f(v)\quad & \mbox{if\,\, ${\cal I}_f(v)+{\cal J}_f(v)=0$} 
 \end{array} \right.  $$
From (\ref{3.1}) we see that $g=f$ a.e. on ${\bR}$ and so $g\in L^1_2({\bRn})$ and 
 $0\le g\le 1$ on ${\bRn}$. Thus by Lemma \ref{lemma2.5} we conclude that 
 ${\cal I}_f\equiv {\cal I}_g$, ${\cal J}_f\equiv {\cal J}_g$.
We need to prove that
$${\cal O}:={\bRn}({\cal I}_g>0)\cap{\bRn}({\cal J}_g>0)={\bRn}.$$
Since ${\cal O}$ is open and non-empty, by translation we can assume that $0\in {\cal O}$, so there 
is $\dt>0$ such that
$B_{\dt}(0)\subset {\cal O}$. Let 
$$\ld=\fr{1}{2}(1+\sqrt{3/2}\,),\quad \eta
=\fr{1}{2}(\sqrt{3/2}\,-1)\dt,$$
$${\cal O}_{\dt}(v)=\{(v_*,\sg)\in {\bRSn}\,\, |\,\,\, |v_*|<\eta,\ v_*\neq v ,
\, \sqrt{1/3}< \cos(\theta/2)<\sqrt{2/3}\,\},\quad v\in{\bRn}$$
where $\theta=\arccos(\langle{v-v_*,\sg}\rangle/|v-v_*|)\in[0,\pi].$  
Using the elementary inequalities
$$|v'|\le \cos(\theta/2)|v|+\sin(\theta/2)|v_*|,  \ \ \ |v_*'|\le \sin(\theta/2)|v|+
\cos(\theta/2)|v_*|$$
one sees that
if $v\in B_{\ld\dt}(0)$ and $(v_*,\sg)\in{\cal O}_{\dt}(v)$, then
$v_*, v',v_*'\in B_{\dt}(0)$.
In fact in this case we have $|v_*|<\eta<\dt$, and
\beas&&
|v'|\le \cos(\theta/2)|v|+\sin(\theta/2)|v_*|
<\sqrt{2/3}\ld \dt+\sqrt{2/3}\eta=\dt
\\
&&
|v_*'|\le \sin(\theta/2)|v|+
\cos(\theta/2)|v_*|<\sqrt{2/3}\ld \dt+\sqrt{2/3}\eta=
\dt.\eeas
So $v_*, v', v_*'\in B_{\dt}(0)$.
Since $0<g={\cal I}_g/({\cal I}_g+{\cal J}_g)<1$
on $B_{\dt}(0)\subset {\cal O}$, this implies
that $g(v')g(v_*')(1-g(v_*))>0$,\, $g(v_*)(1-g(v'))(1-g(v_*'))>0$ for all
$(v_*,\sg)\in{\cal O}_{\dt}(v)$.
Therefore by definition of ${\cal I}_g$ and ${\cal J}_g$ we have
\beas&& {\cal I}_g(v)\ge \intt_{{\cal O}_{\dt}(v)}b(|\la {\bf n},\sg\ra|)g(v')
g(v_*')(1-g(v_*)){\rm d}v_*{\rm d}\sg>0,\\
&& {\cal J}_g(v)\ge \intt_{{\cal O}_{\dt}(v)}b(|\la {\bf n},\sg\ra|)g(v_*)
(1-g(v'))(1-g(v_*')){\rm d}v_*{\rm d}\sg>0\eeas
for all $v\in B_{\ld\dt}(0)$.
Here we have used an obvious fact that
the sets ${\cal O}_{\dt}(v)$ have strictly positive measure with respect to the measure element
${\rm d}v_*{\rm d}\sg$.
Thus $B_{\ld\dt}(0)\subset {\cal O}$. By iteration
we obtain $B_{\ld^n\dt}(0)\subset{\cal O}$, $n=1,2,3, ...,$
and so ${\cal O}={\bRn}$.  By definition of $g$ and that ${\cal I}_f\equiv {\cal I}_g$, ${\cal J}_f\equiv {\cal J}_g$ we conclude that $0<g(v)<1$ for all $v\in {\bRn}$ and
$g$ is continuous on ${\bRn}$. Since $g=f$ a.e. on ${\bRn}$,
it follows that $g$ satisfies Eq.(\ref{1.4}). Thus
$(\fr{g}{1-g})'(\fr{g}{1-g})_*'
=(\fr{g}{1-g})(\fr{g}{1-g})_*$  on
${\bRRSn}$,
and so by a well known result of
Arkeryd (\cite{A},\cite{CIP},\cite{TM}) we conclude  $\fr{g(v)}{1-g(v)}=ae^{-b|v-v_0|^2}$
for all $v\in {\bRn}$ for some constants
$a>0,b>0$ and $v_0\in {\bRn}$.
Thus
$$f(v)=g(v)=F_{a,b}(v)= \fr{ae^{-b|v-v_0|^2}}{1+ ae^{-b|v-v_0|^2}}\qquad {\rm a.e.}\quad v\in{\bRn}.$$
Proof of ``$(\ref{1.11})\Longrightarrow(\ref{1.9})"$: This is obvious since
${\bRn}(f=0), {\bRn}(f=1)$ have measure zero.

Proof of ``$(\ref{1.10})\Longleftrightarrow(\ref{1.11})"$: This is
Lemma \ref{lemma2.3}.

Part (II):  Proof of ``$(\ref{1.12})\Longrightarrow(\ref{1.14})"$:
Suppose that $S(f)=0.$ Since $0\le f\le 1$ on ${\bRn}$, this is equivalent to
 $mes({\bRn}(0<f<1))=0.$ We prove that  $f$ is an
indicator function of a ball.
Let $E={\bRn}(f=1)$. By $0\le f\le 1$ on ${\bRn}$
we have $f(v)={\bf 1}_{E}(v)$ a.e. $v\in{\bRn}$. Also the assumptions $f\in L^1_2({\bRn})$ and $\|f\|_{L^1_2}\neq 0$ imply
$0<mes(E)<\infty$.
Multiplying both sides of
Eq.(\ref{1.4}) by $f={\bf 1}_{E}$  leads to a single equation
$${\bf 1}_{E}(v){\bf 1}_{E}(v_*)[1-{\bf 1}_{E}(v')][1-{\bf 1}_{E}(v_*')]=0 \ \ \
a.e. \ \ (v,v_*,\sg)\in{\bRRS}.$$
So by Theorem \ref{theorem4.2} in Appendix, there is $0<R<\infty$ such that
$f(v)={\bf 1}_{\{|v-v_0|\le R\}}$ a.e. $v\in {\bRn}$.

Proof of ``$(\ref{1.14})\Longrightarrow(\ref{1.12})"$: This is obvious since
$y\log(y)=0$ for $y\in\{0, 1\}$.

Proof of ``$(\ref{1.13})\Longleftrightarrow(\ref{1.14})"$: This is the second conclusion of Lemma \ref{lemma2.2}. 
$\qquad \Box$

\begin{center}{ } \end{center}
\begin{center}\section{ Appendix: A characterization of Euclidean balls }\end{center}

\noindent {\bf 4.1.} Let ${\bRn}=({\bRn},\la \cdot,\cdot\ra)$ be the $n$-dimensional Euclidean
space with the induced norm $|x|=\sqrt{\la x,x\ra}, x\in {\bRn}$. Let $
B_r(x)=\{y\in {\bRn}\,\,|\,\,|y-x|\le r\}$
denote the $n$-dimensional closed ball (or $n$-ball) with the centre
$x\in {\bRn}$ and radius $r>0$, and let ${\bSn}=\{\sg\in {\bRn}\,\,|\,\,|\sg|=1\}$
be the  unit sphere. For any $E\subset {\bRn}$,
denote by $E^{\circ}$, $\p E$ and $\overline{E}$ the interior, boundary and
the closure of $E$ respectively.
It is also denoted by $mes(E)$ the Lebesgue measure
of a Lebesgue measurable set $E\subset {\bRn}$.
  An $n$-dimensional convex body in ${\bRn}$ is defined as a compact convex set
 in ${\bRn}$ with non-empty interior. Throughout this Appendix we assume that $n\ge 2$.

It is well-known that the Euclidean ball can be characterized by
convex bodies of constant width. For instance a result of L.Montejano \cite{M}
(see also \cite{MMO2019})
states that
 {\it if  $2\le k<n$, and $K\subset {\bRn}$ is an $n$-dimensional convex body
having the property that every
$k$-section of $K$ through a point $p_0\in K^{\circ}$ is a convex body of constant width,
then $K$ is a Euclidean $n$-ball.}
A closely related and perhaps more intuitive characterization of balls is as follows which is the first main result of the Appendix:
\begin{theorem}\label{theorem4.1}
Let $n\ge 2,$ let $K\subset {\bRn}$ be a compact set having at least two elements and
satisfying the following condition
\be \forall\,x,y\in \p K,\,\, \forall\,\sg\in {\bSn} \Longrightarrow\,
{\rm either}\,\,\,\fr{x+y}{2}+\fr{|x-y|}{2}\sg\in K\,\,\, {\rm or}\,\,\,
\fr{x+y}{2}-\fr{|x-y|}{2}\sg\in K.\label{4.1}\ee
Then  $K$ is an $n$-dimensional Euclidean ball.
\end{theorem}
The condition (\ref{4.1}) means that for every pair $\{x, y\}\subset \p K$, the
sphere with the centre $\fr{1}{2}(x+y)$ and radius $\fr{1}{2}|x-y|$ has the property that
every pair $\{\fr{1}{2}(x+y)+\fr{1}{2}|x-y|\sg,\,\fr{1}{2}(x+y)-\fr{1}{2}|x-y|\sg\}$ of antipodal points of the sphere meets $K$. The necessity of the condition (\ref{4.1}) is obvious, which is due to the identity
$$\Big|\fr{x+y}{2}+\fr{|x-y|}{2}\sg-x_0\Big|^2+
\Big|\fr{x+y}{2}-\fr{|x-y|}{2}\sg-x_0\Big|^2
=|x-x_0|^2+|y-x_0|^2$$
for all $x,y,x_0\in {\bRn}$ and all $\sg\in {\bSn}$, from which
one sees that for any $r>0$ the ball $K=B_r(x_0)$ satisfies the condition (\ref{4.1}).

The Reuleaux triangle $K_{\rm R}\subset {\mR}^2$, which is a typical example of convex body of constant width which is not a disk, does not satisfies
(\ref{4.1}): let $a, b\in \p K_{\rm R}$ lie in  a symmetric axis of $K_{\rm R}$ with $|a-b|={\rm diam}(K_{\rm R})$
(diameter of $K_{\rm R}$), and consider the middle point $\fr{1}{2}(a+b)$ and
let $\sg\in {\mS}^1$ be perpendicular to $a-b$. Then neither $a':=\fr{1}{2}(a+b)+\fr{1}{2}|a-b|\sg$ nor
$b':=\fr{1}{2}(a+b)-\fr{1}{2}|a-b|\sg$ belongs to $K_{\rm R}$.

Twenty years ago, under a stronger condition that
``$\p K$" in  (\ref{4.1}) is replaced by ``$K$" and assuming that $mes(K)>0$,
 the author \cite{Lu2001} proved that $K$ is a convex body of constant width for all $n\ge 2$, and if $n\ge 3$, then
$K$ is a Euclidean $n$-ball.
The proof in \cite{Lu2001} for $K$ to be a ball relies on the result of L.Montejano \cite{M}
mentioned above. Since the characterization in \cite{M} works
only for $n\ge 3$, one has to look for other methods for dealing with the case of $n=2$.
Observe that the property (\ref{4.1}) holds also for some cross-sections of $K$. Thus once
(\ref{4.1}) as a characterization of balls holds true for $n=2$, the proof for $n\ge 3$ follows from using cross-section and induction argument (see the proof of Theorem \ref{theorem4.1} below).

A measure version as (\ref{4.1}), which initiated our study for (\ref{4.1}), comes
from the collisional kinetic theory. Let $v,v_*\in {\bRn}$ be
the velocities of two particles (with the same mass) just before they collide,
and let $v', v_*'$ be their velocities
just after their collision. Suppose the collision conserves the momentum and kinetic energy, i.e. $v'+v_*'=v+v_*,
|v'|^2+|v_*'|^2=|v|^2+|v_*|^2.$  Then  $v', v_*'$ can be expressed as
\be v'=\fr{v+v_*}{2}+\fr{|v-v_*|}{2}\sg,\quad
v_*'=\fr{v+v_*}{2}-\fr{|v-v_*|}{2}\sg,\quad \sg\in {\bSn}.\label{I}\ee
With the notation (\ref{I}), the second main result of the Appendix can be
stated as follows:

\begin{theorem}\label{theorem4.2}
Let $n\ge 2,$ let $E\subset {\bRn}$ be a Lebesgue measurable set with $0<mes(E)<\infty$ and satisfy
\be {\bf 1}_{E}(v){\bf 1}_{E}(v_*)(1-{\bf 1}_{E}(v'))(1-{\bf 1}_{E}(v_*'))
=0\quad {\rm \, a.e.}\quad (v,v_*,\sg)\in {\bRn}\times{\bRn}\times{\bSn}.\label{4.2}\ee
Then  there is a ball $K=B_R(v_0)\subset {\bRn}$ such that
$mes\big((E\setminus K)\cup (K\setminus E)\big)=0$. In other words, up to a set of measure zero,
$E$ is an $n$-dimensional Euclidean ball.
\end{theorem}
Here  `` a.e." (= for `` almost every") is taken with respect to the product measure on
${\bRn}\times{\bRn}\times{\bSn}$, i.e. (\ref{4.2}) means
\be \inttt_{{\bRn}\times{\bRn}\times {\bSn}}
{\bf 1}_{E}(v){\bf 1}_{E}(v_*)(1-{\bf 1}_{E}(v'))(1-{\bf 1}_{E}(v_*'))
{\rm d}v{\rm d}v_*{\rm d}\sg=0\label{EE}\ee
where ${\rm d}\sg$ is the usual area element of the unit sphere ${\bSn}$.

It is proved in \cite{Lu2001} (for $n=3$)
that if a Fermi-Dirac particle system is at the lowest temperature, then the density
function of the corresponding distribution is an equilibrium and this equilibrium can only be an indicator function $v\mapsto {\bf 1}_{E}(v)$ of a Lebesgue measurable set $E\subset {\bRn}$ (with
 $0<mes(E)<\infty$) that satisfies the equation (\ref{4.2}) (or equivalently the equation (\ref{EE})),
and thus, up to a set of measure zero, $E$ is a ball. However, as mentioned above, the method in \cite{Lu2001} fails for
$n=2$. Now Theorem \ref{theorem4.2} shows that the uniqueness of the equation (\ref{4.2})
holds true for all $n\ge 2$. Thus, as a direct application of Theorem \ref{theorem4.2}, we also finish the classification of equilibrium states for all $n\ge 2$ for the Fermi-Dirac particle system as investigated in \cite{Lu2001}.

The proofs of Theorem \ref{theorem4.1}, Theorem \ref{theorem4.2}
are given in {\bf 4.3}. Before doing that
we first prove in {\bf 4.2} some elementary properties and a relevant basic  lemma (Lemma \ref{lemma5.5}).
\\

\noindent {\bf 4.2.}
\, For any $E\subset {\bRn}$, let
$E^c={\bRn}\setminus E$ be the complement of $E$ with respect to ${\bRn}$.
Recall that $E^{\circ}$ denotes the interior of $E$.
For any ${\bf n}\in {\bSn}$ we denote by $(E)_{\bf n}^{+}$ the ``upper half " of $E$ along the direction ${\bf n}$, i.e.
$$(E)_{\bf n}^{+}=\{x\in E\,\,|\,\, \la x,{\bf n}\ra\ge 0\}.$$

\begin{lemma}\label{lemma5.1}
Let $K\subset {\bRn}$ be a compact set with $K^{\circ}\neq\emptyset$ and suppose without loss of generality that $B_r:=B_r(0)\subset K$
is a largest closed ball inside $K$.
Then
$$(\p B_r)_{\bf n}^+\cap \p K\neq \emptyset\qquad \forall\, {\bf n}\in {\bSn}.$$
\end{lemma}

\begin{proof} Suppose to the contrary that
$(\p B_r)_{\bf n}^+\cap \p K=\emptyset$ for some ${\bf n}\in {\bSn}$.
Then $(\p B_r)_{\bf n}^+\subset K^{\circ}$ and so
${\rm dist}((\p B_r)_{\bf n}^+, (K^{\circ})^c)>0$. Let $0<\dt<\fr{1}{2}{\rm dist}((\p B_r)_{\bf n}^+, (K^{\circ})^c)$
and choose $\vep>0$ such that $2r\vep+\vep^2<\dt^2.$
We prove that
$B_{r+\vep}(\dt{\bf n})\subset K^{\circ}$ which will contradicts
the maximality of $B_r=B_{r}(0)$ in $K$.
Take any $x\in B_{r+\vep}(\dt{\bf n})$. If $|x|<r$, then $x\in K^{\circ}$
because $B_r^{\circ}\subset K^{\circ}$. Suppose that $|x|\ge r$. In this case we
claim that $\la x,{\bf n}\ra> 0$, for otherwise, $\la x,{\bf n}\ra\le 0$, then
$|x|^2
\le |x-\dt{\bf n}|^2-\dt^2
 \le (r+\vep)^2-\dt^2<r^2$  which contradicts
 $|x|\ge r$. So we must have $\la x,{\bf n}\ra >0$. From this and $|x|\ge r$ we deduce $\fr{x}{|x|} r\in (\p B_r)_{\bf n}^+$
and $|x-\fr{x}{|x|} r|
=|x|-r\le |x-\dt{\bf n}|+|\dt{\bf n}|-r
\le \vep+\dt<2\dt$.
Now if $x\not\in K^{\circ}$, i.e. if $x\in
(K^{\circ})^{c}$, then we obtain a contradiction:
$2\dt<{\rm dist}((\p B_r)_{\bf n}^+, (K^{\circ})^c)
\le |\fr{x}{|x|} r-x|<2\dt$. Thus we must
have $x\in K^{\circ}$. This proves that
$ B_{r+\vep}(\dt{\bf n})\subset K^{\circ}$.
\end{proof}
\vskip3mm

The rest of this section is concerned with measure and integration on the sphere. First we recall the
following formula:
for any nonnegative Borel measurable functions $f, g$ on
${\bSn}, [-1,1]$ respectively, and for any $\og\in {\bSn}$,
\bes&& \int_{{\bSn}}f(\sg){\rm d}\sg=\int_{0}^{\pi}\sin^{n-2}(\theta)
{\rm d}\theta\int_{{\mS}^{n-2}(\og)}f\big(\cos(\theta)\og+\sin(\theta)\tilde{\sg}\big){\rm d}\tilde{\sg}, \label{Int1}\\
&&
 \int_{{\bSn}}g(\la \og, \sg\ra){\rm d}\sg=|{\mS}^{n-2}|
\int_{0}^{\pi}\sin^{n-2}(\theta)g(\cos(\theta))
{\rm d}\theta \dnumber \label{Int10}\ees
where, for $n\ge 3$,
${\mS}^{n-2}(\og):=\{\tilde{\sg}\in {\bSn}\,\,|\,\,
\la \tilde{\sg},\og\ra=0\}$ and
${\rm d}\tilde{\sg}$ is the $n-2$-dimensional area element on the
${\mS}^{n-2}(\og)$, and for the case $n=2$ we define
${\mS}^{0}(\og):=\{\og_{\bot},\, -\og_{\bot}\}$ where
$\og_{\bot}\in {\mS}^1$ satisfies $\la \og_{\bot}, \og\ra=0$;
the integral on ${\mS}^{0}(\og)$ is defined by
$$\int_{{\mS}^{0}(\og)}f(\tilde{\sg}){\rm d}\tilde{\sg}=
f(\og_{\bot})+f(-\og_{\bot}).$$
In this Appendix we denote the areas of ${\bSn}, {\mS}^{n-2}(\og)$ by
$|{\bSn}|, |{\mS}^{n-2}(\og)|=|{\mS}^{n-2}|$, and for the case $n=2$, we
define $|{\mS}^{0}(\og)| =|{\mS}^{0}|=2.$

\noindent For any $a,b,t\in {\mR}$ we denote $a\wedge b=\min\{a,b\},\, a\vee b=\max\{a, b\},\, (t)_{+}= t\vee 0$.

\begin{lemma}\label{lemma5.4} Let $F, f$ be nonnegative Borel functions on $[-1,1]\times {\bRRn}$ and
${\bRn}$ respectively. Then
\bes&&\intt_{{\bRSn}}F\big(\la{\bf n},\sg\ra, v_*', v_*\big){\rm d}v_*{\rm d}\sg
\nonumber\\
&&=2^{n-1}\int_{{\bRn}}{\rm d}v_*
\int_{0}^{\pi/2}\fr{\sin^{n-2}(\theta)}{\cos^{2}(\theta)}
{\rm d}\theta\int_{{\mS}^{n-2}({\bf
n})}F\big(\cos(2\theta),\,v_*,\,
v_*+|v-v_*|\tan(\theta){\tilde \sg}\big){\rm d}{\tilde \sg},\qquad \dnumber
\label{5.3}\\ \nonumber\\
&& \fr{1}{r^n}
\int_{B_{r}(v)}
f(v_*){\rm d}v_*\nonumber
\\
&&\ge
\intt_{{\bSn}\times{\bSn}}\fr{\sqrt{1-\la \og,\sg\ra^2}}{2^{n+1}|{\mS}^{n-2}|}
\Big(f\big(v+\fr{r}{2}(\og+
\sg)\big)+f\big(v+\fr{r}{2}(\og-
\sg)\big)\Big)
{\rm d}\sg{\rm d}\og\dnumber \label{5.4}\ees
for all $v\in {\bRn}$ and all $r>0$, where ${\bf n}=(v-v_*)/|v-v_*|, v_*'=
\fr{1}{2}(v+v_*)-\fr{1}{2}|v-v_*|\sg$.
\end{lemma}

\begin{proof}
Making change of variables $v_*=v-z, z=r\og\,\Longrightarrow\, {\rm d}v_*=
r^{n-1}{\rm d}r{\rm d}\og$, then letting $\og=\cos(\theta)\sg+\sin(\theta)\tilde{\sg},\tilde{\sg}
\in {\mS}^{n-2}(\sg)\,\Longrightarrow\, {\rm d}\og=\sin^{n-2}(\theta){\rm d}\theta {\rm d}
\tilde{\sg}$,  and then changing $\theta\to 2\theta$,
$r\to\fr{r}{\cos(\theta)}$, we compute
\beas&& \intt_{{\bRSn}}F\big(\la{\bf n},\sg\ra, v_*', v_*\big){\rm d}v_*{\rm d}\sg
=2^{n-1}\int_{{\bSn}}{\rm d}\sg\int_{0}^{\infty}r^{n-1}{\rm d}r \int_{0}^{\pi}{\bf 1}_{\{\cos(\theta)>0\}}\fr{\sin^{n-2}(\theta)}{\cos^2(\theta)}{\rm d}\theta
\\
&&\times\int_{{\mS}^{n-2}(\sg)} F\Big(\cos(2\theta),
v-r\big(\cos(\theta)\sg+\sin(\theta)\tilde{\sg}\big),
v+\fr{r}{\cos(\theta)}\sg -2r\big(\cos(\theta)\sg+\sin(\theta)\tilde{\sg}\big)\Big){\rm d}\tilde{\sg}
\\
&&=2^{n-1}\int_{{\bRn}}{\rm d}v_*\int_{{\bSn}}
{\bf 1}_{\{\la {\bf n},\sg\ra>0\}}\fr{1}{\la{\bf n},\sg\ra^2} F\Big(2\la
{\bf n},\sg\ra^2-1, v_*,
v+\fr{|v-v_*|}{\la{\bf n},\sg\ra}\sg -2(v-v_*)\Big){\rm d}\sg
\\
&&=2^{n-1}\int_{{\bRn}}{\rm d}v_*\int_{0}^{\pi}{\bf 1}_{\{\cos(\theta)>0\}}\fr{\sin^{n-2}(\theta)}{\cos^2(\theta)}
{\rm d}\theta\int_{{\mS}^{n-2}({\bf n})}F\big(\cos(2\theta), v_*,
v_*+|v-v_*|\tan(\theta)\tilde{\sg}\big){\rm d}\tilde{\sg}
\eeas
which is equal to the right hand side of (\ref{5.3}).

To prove (\ref{5.4}), we first prove that
 for any nonnegative Borel measurable functions $h, g$ on $(0,1)$
 and ${\bSn}$
 respectively, it holds
\bes&& \int_{{\bSn}}|\la \og,\sg\ra|^{n-2}h(|\la \og,\sg\ra|)
g\big(\og-2\la\og,\sg\ra\sg\big)
{\rm d}\sg\nonumber \\
&&=\fr{1}{2^{n-1}}\int_{{\bSn}}\Big(h(\sin(\theta/2))g(\sg)+
h(\cos(\theta/2))g(-\sg)\Big)\Big|_{\theta=\arccos(\la \og,\sg\ra)}
{\rm d}\sg.
\label{Int2}\ees
To do this we observe that the mapping $\sg\mapsto
 \la \og,\sg\ra\sg$ is even and so using the formula (\ref{Int1}) we compute
\beas&& \int_{{\bSn}}|\la \og,\sg\ra|^{n-2}h(|\la \og,\sg\ra|)
g\big(\og-2\la\og,\sg\ra\sg\big)
{\rm d}\sg=2\int_{{\bSn}}{\bf 1}_{\{\la\og,\sg\ra>0\}}\{\cdots\}{\rm d}\sg\\
&&=2\int_{0}^{\pi/2}\sin^{n-2}(\theta)\cos^{n-2}(\theta)h(\cos(\theta)){\rm d}\theta
\int_{{\mS}^{n-2}(\og)}g\big(-\cos(2\theta)\og-\sin(2\theta)\tilde{\sg}\big)\big)
{\rm d}\tilde{\sg}.
\eeas
Then changing variable $\theta\to \theta/2$ and using the identity
$\int_{{\mS}^{n-1}}\vp(-\sg){\rm d}\sg
=\int_{{\mS}^{n-1}}\vp(\sg){\rm d}\sg$ it is
easily shown that (\ref{Int2}) holds true.  From (\ref{Int2})
we obtain
\bes&& \int_{{\bSn}}|\la \og,\sg\ra|^{n-2}h(|\la \og,\sg\ra|)
g\big(\og-2\la\og,\sg\ra\sg\big)
{\rm d}\sg\nonumber \\
&&\ge \fr{1}{2^{n-1}}\int_{{\bSn}}\big[h(\sin(\theta/2))\wedge
h(\cos(\theta/2))\big]
\big(
g(\sg)+g(-\sg)\big)\Big|_{\theta=\arccos(\la \og,\sg\ra)}
{\rm d}\sg.
\label{Int3}\ees
Now we choose $h(t)=t/(\sqrt{1-t^2})^{(n-3)}=t^{n-2}/(t\sqrt{1-t^2})^{n-3}, t\in (0,1)$, and  write
$$v+r\la \og, \sg\ra\sg=
v+\fr{r}{2}\big(\og-(\og-2\la \og, \sg\ra\sg
)\big).$$
Then we deduce from (\ref{Int3}) and $\cos(\theta)=\la \og,\sg\ra,\theta\in [0,\pi],$ that
\beas&&
\int_{{\bSn}}\fr{|\la \og,\sg\ra|^{n-1}}{\big(\sqrt{1-\la \og,\sg\ra^2\,}
\big)^{n-3}}
f\big(v+r\la \og, \sg\ra\sg\big){\rm d}\sg\\
&&
\ge
\int_{{\bSn}}
\fr{\sin^{n-2}(\theta/2)\wedge \cos^{n-2}(\theta/2)}{2^{n-1}\sin^{n-3}(\theta/2)\cos^{n-3}(\theta/2)}
\Big(f\big(v+\fr{r}{2}(\og+\sg)\big)+f\big(v+\fr{r}{2}(\og-\sg)\big)
\Big)
{\rm d}\sg\\
&&\ge \int_{{\bSn}}\fr{\sin(\theta)}{2^n}
\Big(f\big(v+\fr{r}{2}(\og+\sg)\big)+f\big(v+\fr{r}{2}(\og-\sg)\big)
\Big)
{\rm d}\sg.\eeas
On the other hand, taking integration for $\og\in {\bSn}$
and noting again that $\sg\mapsto \la \og, \sg\ra\sg$ is even, we compute using
the formula (\ref{Int1}) that
\beas&& \intt_{{\bSn}\times{\bSn}}
\fr{|\la \og,\sg\ra|^{n-1}}{\big(\sqrt{1-\la \og,\sg\ra^2\,}
\big)^{n-3}}
f\big(v+r\la \og, \sg\ra\sg\big){\rm d}\sg{\rm d}\og=
2\intt_{{\bSn}\times {\bSn}}{\bf 1}_{\{\la\og,\sg\ra>0\}}\{\cdots\}{\rm d}\sg{\rm d}\og\\
&&=2|{\mS}^{n-2}|
\int_{{\bSn}}{\rm d}\sg\int_{0}^{\pi/2}\sin(\theta)
\cos^{n-1}(\theta)
f\big(v+r\cos(\theta)\sg\big){\rm d}\theta
=\fr{2|{\mS}^{n-2}|}{r^n}\int_{B_{r}(v)}
f(v_*){\rm d}v_*.
\eeas
This gives (\ref{5.4}).
\end{proof}
\vskip2mm

As mentioned above, in this Appendix we use $|S|$ to denote the
$n-1$-dimensional Hausdorff measure of a Hausdorff measurable subset
$S\subset {\bSn}$. When dealing with measure and measurability, two things may be reminded:

(1)
If a Lebesgue measurable set $Z\subset {\mR}^N$, a Hausdorff measurable set $S\subset {\bSn}$, etc.,
satisfy $mes(Z)=0$, $|S|=0$, etc., then they are sometimes simply called
{\it null } sets.

(2)
Since for instance the Lebesgue measure is regular, a Lebesgue measurable set (with finite measure) under consideration can be replaced by one of its Borel subset with the same measure
when dealing with measurability on lower dimensional sets. In other words, for notational convenience
we will do not distinguish between two measurable sets that differ by a null set.

For any Borel set $E\subset {\bRn}$ and $v,v_*\in {\bRn}$, we define a Borel subset of ${\bSn}$ by
$${\mS}^{n-1}_{v,v_*}(E)=\Big\{\sg\in {\bSn}\,\,\big|\,\,\,
\fr{v+v_*}{2}+\fr{|v-v_*|}{2}\sg\in E\,\,\,\,{\rm or}\,\,\,\, \fr{v+v_*}{2}-\fr{|v-v_*|}{2}\sg\in E\Big\}.$$
For the set $E$ in Theorem \ref{theorem4.2}, after a modification on a null set we may assume that
$E$ is a Borel set, so that
the condition (\ref{4.2}) is equivalent to
$|{\mS}^{n-1}_{v,v_*}(E)|=|{\mS}^{n-1}|$ for almost all $(v,v_*)\in E\times E$. The following lemma deals with a general case; it also shortens the proof of Theorem \ref{theorem4.2}.

\begin{lemma}\label{lemma5.5}
Let $E\subset {\bRn}$ be a Borel set with $0<mes(E)<\infty$ satisfying for some
constant $0<\ld\le 1$
\be |{\mS}^{n-1}_{v,v_*}(E)|\ge \ld|{\mS}^{n-1}|\qquad \,\,\,\,{\rm a.e.}
\quad (v,v_*)\in E\times E.\label{5.7}\ee
Then there exists a subset $Z_0\subset E$ with $mes(Z_0)=0$
such that $E\setminus Z_0$ is a bounded set and the compact set $K=\overline{E\setminus Z_0}$ satisfies
$mes(K^{\circ}\setminus E)=0$ and \be |{\mS}^{n-1}_{v,v_*}(K)|\ge \ld|{\mS}^{n-1}|\qquad \forall\, (v,v_*)\in K\times K.\label{5.8}\ee
Furthermore if $mes(\p K)=0$, then
 $mes\big((K\setminus E)\cup (E\setminus K)\big)=0$.
\end{lemma}

{\bf Remarks.} (1) It is not clear whether (\ref{5.8}) already
 implies $mes\big((K\setminus E)\cup (E\setminus K)\big)=0$.

(2) There are many compact but non-convex sets $K$ that
satisfy (\ref{5.8}) for $0<\ld<1$ with $\ld\approx 1$. A typical example is the
annulus
$K_{\vep}=\{x\in {\bRn}\,|\, \vep\le |x|\le 1\}, 0<\vep<1.$
It is easily proven that $K_{\vep}$ satisfies (\ref{5.8}) with
$$\ld\ge  \int_{0}^{\fr{1-\vep^2}{1+\vep^2}} (1-t^2)^{\fr{n-3}{2}}{\rm d}t\Big/\int_{0}^{1} (1-t^2)^{\fr{n-3}{2}}{\rm d}t\to 1\quad (\vep\to 0).$$

{\bf Proof of Lemma} \ref{lemma5.5}.
\, {\it Step1}. We first prove that for any Borel subset $\wt{E}\subset E$ satisfying
$mes(E\setminus \wt{E})=0$, the set $\wt{E}$ also satisfies condition
(\ref{5.7}).
Using the notation (\ref{I}) we have
\be {\bf 1}_{{\mS}^{n-1}_{v,v_*}(E)}(\sg)={\bf 1}_{E}(v')\vee {\bf 1}_{E}(v_*')
\qquad \forall\, (v,v_*,\sg)\in {\bRn}\times {\bRn}\times {\bSn}.\label{5.9}\ee
So $|{\mS}^{n-1}_{v,v_*}(E)|=\int_{{\bSn}}\big({\bf 1}_{E}(v')\vee {\bf 1}_{E}(v_*')\big) {\rm d}\sg$ and so the condition
(\ref{5.7}) is equivalent to
\be \intt_{E\times E}\left(\ld |{\bSn}|-\int_{{\bSn}}\big({\bf 1}_{E}(v')\vee {\bf 1}_{E}(v_*')\big) {\rm d}\sg\right)_{+}
{\rm d}v{\rm d}v_*=0.\label{5.10}\ee
Take any Borel subset $\wt{E}\subset E$ satisfying
$mes(E\setminus \wt{E})=0$. Let $Z=E\setminus \wt{E}$, then
${\bf 1}_{E}(v')\le {\bf 1}_{\wt{E}}(v')+{\bf 1}_{Z}(v')$, etc. so that
${\bf 1}_{E}(v')\vee {\bf 1}_{E}(v_*')\le {\bf 1}_{\wt{E}}(v')\vee {\bf 1}_{\wt{E}}(v_*')
+{\bf 1}_{Z}(v')\vee {\bf 1}_{Z}(v_*')$
and so
$\ld|{\bSn}|-\int_{{\bSn}}\big({\bf 1}_{\wt{E}}(v')\vee {\bf 1}_{\wt{E}}(v_*')\big) {\rm d}\sg
\le \ld|{\bSn}|-\int_{{\bSn}}\big({\bf 1}_{E}(v')\vee {\bf 1}_{E}(v_*')\big) {\rm d}\sg
+\int_{{\bSn}}\big({\bf 1}_{Z}(v')\vee {\bf 1}_{Z}(v_*')\big) {\rm d}\sg$ and so
\beas&&\left(\ld |{\bSn}|-\int_{{\bSn}}\big({\bf 1}_{\wt{E}}(v')\vee {\bf 1}_{\wt{E}}(v_*')\big) {\rm d}\sg\right)_{+}
\le\left(\ld |{\bSn}|-\int_{{\bSn}}\big({\bf 1}_{E}(v')\vee {\bf 1}_{E}(v_*')\big) {\rm d}\sg\right)_{+}
 \\
 &&+\int_{{\bSn}}\big({\bf 1}_{Z}(v')\vee {\bf 1}_{Z}(v_*')\big) {\rm d}\sg.\eeas
Then taking integration  over $(v,v_*)\in \wt{E}\times \wt{E}\subset E\times E$ and notice that
$$\inttt_{{\bRRSn}}\big({\bf 1}_{Z}(v')\vee {\bf 1}_{Z}(v_*')\big)
{\rm d}v{\rm d}v_* {\rm d}\sg=
|{\bSn}|\intt_{{\bRRn}}\big({\bf 1}_{Z}(v)\vee {\bf 1}_{Z}(v_*)\big)
{\rm d}v{\rm d}v_*=0$$
we see that $\wt{E}$ satisfies (\ref{5.10}), or equivalently,
$\wt{E}$ satisfies (\ref{5.7}).

Let $E_1=\{v\in E\,|\,\lim\limits_{r\to 0^+}mes(E\cap B_{r}(v))/mes(B_{r}(v))=1\}$
be the set of all density points of $E$. By Lebesgue's density theorem we know that
$mes(E\setminus E_1)=0$. By the regularity of the Lebesgue measure and the property just proved above
we may assume without loss of generality that $E$ is a Borel set and every point in
$E$ is a density point of $E$.
Let
\be {\cal Z}_{E}=\big\{(v,v_*)\in E\times E\,\,\,\big|\,\,\,|{\mS}^{n-1}_{v,v_*}(E)|< \ld |{\bSn}|\big\}.
\label{5.12}\ee
By assumption (\ref{5.7}),
${\cal Z}_{E}$ is a null set in ${\mR}^{2n}$.
We next prove the following inequalities:
\be
\int_{{\bSn}}\big(1-\la \og,\sg\ra^2\big)
{\bf 1}_{{\mS}^{n-1}_{v,v_*}(E)}(\sg){\rm d}\sg\ge C_{n}|{\mS}^{n-1}_{v,v_*}(E)|^{5}
\qquad \forall\,(v,v_*,\og)\in {\bRRSn}\qquad  \label{5.13}\ee
\be\int_{{\bSn}}\big(1-\la \og,\sg\ra^2\big)
{\bf 1}_{{\mS}^{n-1}_{v,v_*}(E)}(\sg){\rm d}\sg\ge C_{n,\ld}
\qquad \forall\, (v,v_*)\in (E\times E)\setminus {\cal Z}_{E},\,\,\forall\, \og\in {\bSn}\quad\qquad \label{5.14}\ee
where $C_{n}=\big(|{\mS}^{n-2}|
\int_{0}^{\pi}\sin^{n-5/2}(\theta){\rm d}\theta\big)^{-4},
C_{n,\ld}=C_{n}(\ld|{\bSn}|)^{5}$.
In fact for any $(v,v_*,\og)\in {\bRRSn}$, using
H\"{o}lder inequality  we have
\beas&&|{\mS}^{n-1}_{v,v_*}(E)|=
\int_{{\bSn}}{\bf 1}_{{\mS}^{n-1}_{v,v_*}(E)}(\sg){\rm d}\sg
\\
&&\le\left(
\int_{{\bSn}}\big(1-\la \og,\sg\ra^2\big)\,{\bf 1}_{{\mS}^{n-1}_{v,v_*}(E)}(\sg)
{\rm d}\sg\right)^{1/5}
\left(
\int_{{\bSn}}\big(1-\la \og,\sg\ra^2\big)^{-1/4}{\rm d}\sg\right)^{4/5}\eeas
which gives (\ref{5.13}). The inequality (\ref{5.14}) follows from
(\ref{5.13}) and the definition of ${\cal Z}_{E}$.

{\it Step2.} In this step we prove that there is a null set $Z_0\subset E$ such that
$E\setminus Z_0$ is a bounded Borel set and the compact set $K=\overline{E\setminus Z_0}$ satisfies
(\ref{5.8}).

Since the set
${\cal Z}_{E}$ defined in (\ref{5.12})
has measure zero, it follows from Fubini theorem and the regularity of the Lebesgue measure that
there is a null set $Z_0\subset E$, such that $E\setminus Z_0$ is a Borel set and
for every $v\in E\setminus Z_0$ there is a null set $Z_v\subset E$ such that $(v,v_*)\in (E\times E)\setminus {\cal Z}_{E}$ for all $v_*\in E\setminus Z_v$.
Given any $v\in E\setminus Z_0$. We now use
Lemma \ref{lemma5.4} to find a relation between $mes(E)$ and $|v|$. In the
integral identity (\ref{5.3}) let us choose
$F(\la{\bf n},\sg\ra, v_*', v_*)=
(1-\la{\bf n},\sg\ra^2){\bf 1}_{E}(v_*')(\fr{|v-v_*|}{|v_*|})^{1/2}$. Then for all $v\in E\setminus Z_0$
\beas&&\intt_{{\bRSn}}(1-\la{\bf n},\sg\ra^2){\bf 1}_{E}(v_*')\Big(\fr{|v-v_*|}{|v_*|}\Big)^{1/2}
{\rm d}v_*{\rm d}\sg\\
&&=2^{n+1}\int_{{\bRn}}{\bf 1}_{E}(v_*){\rm d}v_*
\int_{0}^{\pi/2}\sin^{n}(\theta){\rm d}\theta\int_{{\mS}^{n-2}({\bf
n})}\Big(
\fr{|{\bf n}-\tan(\theta){\tilde \sg}|}{\big|
|v-v_*|^{-1}v_*+\tan(\theta){\tilde \sg}\big|}\Big)^{1/2}{\rm d}{\tilde \sg}.
\\
&&\le 2^{n+1}|{\mS}^{n-2}|\int_{E}{\rm d}v_*
\int_{0}^{\pi/2}
\fr{[\tan(\theta)]^{1/2}}{|a-\tan(\theta)|^{1/2}}{\rm d}\theta
\le C^*_n mes(E)
\eeas
where $a=|v-v_*|^{-1}|v_*|, 0<C^*_n<2^{n+4}|{\mS}^{n-2}|$, and we have used $\tilde{\sg}\bot {\bf n}
\Longrightarrow\, |{\bf n}-\tan(\theta){\tilde \sg}|=1/\cos(\theta)$ and that
$\int_{0}^{\pi/2}
\fr{[\tan(\theta)]^{1/2}}{|a-\tan(\theta)|^{1/2}}{\rm d}\theta=
\int_{0}^{\infty}
\fr{1}{|a- t|^{1/2}}\fr{t^{1/2}}{1+t^2}{\rm d}t<7.$ On the other hand
using (\ref{5.9}) we have
\be{\bf 1}_{E}(v')+{\bf 1}_{E}(v_*')\ge {\bf 1}_{{\mS}^{n-1}_{v,v_*}(E)}(\sg)\quad \forall\,(v,v_*,\sg)\in {\bRn}\times {\bRn}\times {\bSn}.\label{5.15}\ee
and so we deduce for all $v\in E\setminus Z_0$
\beas&&\intt_{{\bRSn}}(1-\la{\bf n},\sg\ra^2){\bf 1}_{E}(v_*')\Big(\fr{|v-v_*|}{|v_*|}\Big)^{1/2}
{\rm d}v_*{\rm d}\sg
\\
&&=\fr{1}{2}\int_{{\bRn}}\Big(\fr{|v-v_*|}{|v_*|}\Big)^{1/2}
{\rm d}v_*\int_{{\bSn}}
(1-\la{\bf n},\sg\ra^2)\big({\bf 1}_{E}(v_*')+{\bf 1}_{E}(v_*')
\big){\rm d}\sg
\\
&&\ge \fr{1}{2}\int_{E}\Big(\fr{|v-v_*|}{|v_*|}\Big)^{1/2}
{\rm d}v_*\int_{{\bSn}}
(1-\la{\bf n},\sg\ra^2) {\bf 1}_{{\mS}^{n-1}_{v,v_*}(E)}(\sg){\rm d}\sg
\ge \fr{1}{2}C_{n,\ld}\int_{E}\fr{|v-v_*|^{1/2}}{|v_*|^{1/2}}
{\rm d}v_*\eeas
where we have used (\ref{5.14}).
Connecting the above estimates we obtain
\be C^*_n mes(E)\ge \fr{1}{2}C_{n,\ld}\int_{E}\fr{|v-v_*|^{1/2}}{|v_*|^{1/2}}
{\rm d}v_*\qquad \forall\, v\in E\setminus Z_0.\label{5.16}\ee
Since
$|v-v_*|^{1/2}\ge |v|^{1/2}-|v_*|^{1/2}$ and
$0<\int_{E}\fr{1}{|v_*|^{1/2}}
{\rm d}v_*<\infty$, it follows from (\ref{5.16}) and $C_{n,\ld}>0$ that the set
$E\setminus Z_0$ is bounded.

We next prove that the compact set $K:=\overline{E\setminus Z_0}$ satisfies (\ref{5.8}). Since $Z_0$ is a null set, the set $E\setminus Z_0$
also satisfies (\ref{5.7}), hence the set ${\cal Z}_{E\setminus Z_0}\subset {\mR}^{2n}$ defined in (\ref{5.12}) has measure zero.
Since $mes(E\setminus Z_0)=mes(E)>0$ and since every point in
$E$ is a density point of $E$, this implies that  every point in
$E\setminus Z_0$ is a density point of $E\setminus Z_0$.
It follows that the set $[(E\setminus Z_0)\times (E\setminus Z_0))]\setminus {\cal Z}_{E\setminus Z_0}$
is dense in $K\times K$:  given any $(v,v_*)\in K\times K$, there is a sequence
$\{(v_k, v_{*k})\}_{k=1}^{\infty}\subset [(E\setminus Z_0)\times (E\setminus Z_0))]\setminus
{\cal Z}_{E\setminus Z_0}$
such that $ (v_k, v_{*k})\to (v,v_*)\,(k\to\infty)$.
Then, since $(v,v_*,\sg)\mapsto v'=\fr{1}{2}(v+v_*)+\fr{1}{2}|v-v_*|\sg,
(v,v_*,\sg)\mapsto v_*'=\fr{1}{2}(v+v_*)-\fr{1}{2}|v-v_*|\sg$ are
continuous and since $K^c={\bRn}\setminus K$ is open and $K^c\subset (E\setminus Z_0)^c$,
it follows that
$${\bf 1}_{{\mS}_{v,v_*}^{n-1}(K)}(\sg)\ge \limsup_{k\to\infty}
{\bf 1}_{{\mS}_{v_k,v_{*k}}^{n-1}(E\setminus Z_0)}(\sg)\qquad \forall\,\sg\in {\bSn}$$
and thus by Fatou's lemma
we obtain
\beas&&|{\mS}_{v,v_*}^{n-1}(K)|=\int_{{\bSn}}{\bf 1}_{{\mS}_{v,v_*}^{n-1}(K)}(\sg){\rm d}\sg\ge
\int_{{\bSn}}\limsup_{k\to\infty}
{\bf 1}_{{\mS}_{v_k,v_{*k}}^{n-1}(E\setminus Z_0)}(\sg){\rm d}\sg\\
&&\ge \liminf_{k\to\infty}\int_{{\bSn}}
{\bf 1}_{{\mS}_{v_k,v_{*k}}^{n-1}(E\setminus Z_0)}(\sg)
{\rm d}\sg
=
\liminf_{k\to\infty}|{\mS}_{v_k,v_{*k}}^{n-1}(E\setminus Z_0)|\ge \ld|{\bSn}|.\eeas
This proves that $K$ satisfies (\ref{5.8}).

{\it Step3.} Prove that $mes(K^{\circ}\setminus E)=0.$  We may assume that $K^{\circ}$ is non-empty.
It suffices to prove that there is a
constant $\alpha>0$, e.g.
$\alpha=\fr{nC_{n,\ld}}{4^{n+1}|{\mS}^{n-2}|}\ld$, such that
for any $v_0\in K^{\circ}$ and any $R>0$ satisfying $B_R(v_0)\subset K^{\circ}$ we have
\be mes(E\cap B_R(v_0))\ge \alpha mes(B_R(v_0)).\label{5.17}\ee
In fact if this holds true, then for any $v\in K^{\circ}\setminus E$ and
any $r>0$ satisfying $B_r(v)\subset K^{\circ}$ we have
\beas mes\big((K^{\circ}\setminus E)\cap B_r(v)\big)
\le (1-\alpha)mes(B_r(v)).\eeas
This implies that $mes(K^{\circ}\setminus E)=0$.

Now given any $v_0\in K^{\circ}$ and any $R>0$ satisfying $B_R(v_0)\subset K^{\circ}$.
We first prove that there is a null set
$Z_1\subset E\times (0,\infty)$ such that for all $(x,r)\in \{(E\cap B_{R/6}(v_0))\times [R/2, 5R/6]\}\setminus Z_1
$ and all $\sg\in {\bSn}$
\be \fr{1}{R^n}mes(E\cap B_{R}(v_0))\ge \fr{C_{n,\ld}}{4^{n+1}|{\mS}^{n-2}|}|{\mS}^{n-1}_{x+r\sg, x-r\sg}(E)|.\label{5.18*}
\ee
We will use a fact that for any open set $G\subset {\bRn}$
satisfying $mes(K\cap G)>0$, it holds $mes(E\cap G)>0$. This is because $E\setminus Z_0$ is dense in $K$ and every point in $E$ is a density point of $E$. From this we have $mes(E\cap B_{R/6}(v_0))>0$.
Using (\ref{5.15}) we have for all $x\in{\bRn}$ and all $\og,\sg\in {\bSn}$
\be {\bf 1}_{E}\big(x+\fr{r}{2}(\og+\sg)\big)+
{\bf 1}_{E}\big(x+\fr{r}{2}(\og-\sg)\big)\ge {\bf 1}_{{\mS}^{n-1}_{x, x+r\og}(E)}(\sg), \label{5.20*}\ee
\be {\bf 1}_{E}(x
+r\og)+{\bf 1}_{E}(x-r\og)\ge {\bf 1}_{{\mS}^{n-1}_{x+r\sg, x-r\sg}(E)}(\og).\label{5.21*}\ee
By change of variable and Fubini theorem we have
$$\int\!\!\!\int_{E\times(0,\infty)}r^{n-1}{\rm d}x {\rm d}r\int_{{\bSn}}
{\bf 1}_{{\cal Z}_{E}}(x,x+r\og){\rm d}\og=mes({\cal Z}_{E})=0$$
and so there is a null set $Z_1\subset E\times (0,\infty)$ such that
for every $(x, r)\in (E\times (0,\infty))\setminus Z_1$,
there is a null set $S_{x,r}\subset {\bSn}$ such that
${\bf 1}_{{\cal Z}_{E}}(x,x+r\og)=0$ for all $\og\in {\bSn}\setminus S_{x,r}$.
By definition of ${\cal Z}_{E}$, this implies that for every $(x, r)\in (E\times (0,\infty))\setminus Z_1$,
\be |{\mS}^{n-1}_{x,x+r\og}(E)|\ge \ld |{\bSn}|{\bf 1}_{E}(x+r\og)\qquad \forall\, \og\in {\bSn}\setminus S_{x,r}.\label{5.19*}\ee
Let $(x,r)\in \{(E\cap B_{R/6}(v_0))\times [R/2, 5R/6]\}\setminus Z_1
$. Then $B_{r}(x)\subset B_{R}(v_0)$ and by Lemma \ref{lemma5.4},
(\ref{5.20*}), (\ref{5.13}), (\ref{5.19*}) and (\ref{5.21*})  we compute
\beas&&\fr{1}{R^n}mes(E\cap B_{R}(v_0))\ge\fr{1}{(2r)^n}mes(E\cap B_{r}(x))
=
\fr{1}{2^nr^n}\int_{B_{r}(x)}{\bf 1}_{E}(v_*){\rm d}v_*
\\
&&\ge\fr{1}{2^n}
\intt_{{\bSn}\times {\bSn}}\fr{\sqrt{1-\la \og,\sg\ra^2}}{2^{n+1}|{\mS}^{n-2}|}
\Big({\bf 1}_{E}(x+\fr{r}{2}(\og+
\sg)\big)+{\bf 1}_{E}\big(x+\fr{r}{2}(\og-
\sg)\big)\Big)
{\rm d}\sg{\rm d}\og
\\
&&\ge\fr{1}{2^n}
\int_{{\bSn}}{\rm d}\og \int_{{\bSn}}
\fr{\sqrt{1-\la \og,\sg\ra^2}}{2^{n+1}|{\mS}^{n-2}|}
{\bf 1}_{{\mS}^{n-1}_{x,x+r\og}(E)}(\sg)
{\rm d}\sg
\\
&&\ge \fr{C_{n}}{2^{2n+1}|{\mS}^{n-2}|}\int_{{\bSn}}|{\mS}^{n-1}_{x,x+r\og}(E)|^{5}{\rm d}\og
\ge\fr{C_{n}}{2^{2n+1}|{\mS}^{n-2}|}\int_{{\bSn}}(\ld |{\mS}^{n-1}|)^{5}{\bf 1}_{E}(x+r\og){\rm d}\og
\\
&&
=\fr{C_{n,\ld}}{2^{2n+2}|{\mS}^{n-2}|}\int_{{\bSn}}\big({\bf 1}_{E}(x+r\og)+{\bf 1}_{E}(x-r\og)\big){\rm d}\og
\\
&&\ge \fr{C_{n,\ld}}{2^{2n+2}|{\mS}^{n-2}|}\int_{{\bSn}}{\bf 1}_{{\mS}^{n-1}_{x+r\sg, x-r\sg}(E)}(\og){\rm d}\og=\fr{C_{n,\ld}}{4^{n+1}|{\mS}^{n-2}|}|{\mS}^{n-1}_{x+r\sg, x-r\sg}(E)|
\qquad \forall\, \sg\in {\bSn}.\eeas
So (\ref{5.18*}) holds true.

Next, to  estimate $|{\mS}^{n-1}_{x+r\sg, x-r\sg}(E)|$ we use change of variable to see that
$$\intt_{{\bRn}\times(0,\infty)}r^{n-1}{\rm d}x{\rm d}r \int_{{\bSn}}
{\bf 1}_{{\cal Z}_{E}}(x+r\sg,x-r\sg){\rm d}\sg=mes({\cal Z}_E)=0$$
and so there is a null set $Z_2\subset {\bRn}\times(0,\infty)$ such that
for every $(x, r)\in ({\bRn}\times(0,\infty))\setminus Z_2$, there is a null set $S^*_{x,r}\subset {\bSn}$ such that
${\bf 1}_{{\cal Z}_{E}}(x+r\sg,x-r\sg)=0$ for all $\sg\in {\bSn}\setminus S^*_{x,r}.$ This implies
that for all $(x, r)\in ({\bRn}\times(0,\infty))\setminus Z_2$
\be |{\mS}^{n-1}_{x+r\sg, x-r\sg}(E)|\ge \ld |{\mS}^{n-1}|{\bf 1}_{E}(x+r\sg){\bf 1}_{E}(x-r\sg)\qquad \forall\, \sg\in{\bSn}\setminus S^*_{x,r}.\label{5.22}\ee
Connecting (\ref{5.18*}) we then obtain for all $(x,r)\in \{(E\cap B_{R/6}(v_0))\times [R/2,5R/6]\}\setminus (Z_1\cup Z_2)$ and for all $\sg\in{\bSn}\setminus S^*_{x,r}$
$$ \fr{1}{R^n}mes(E\cap B_{R}(v_0))
\ge \fr{C_{n,\ld}}{4^{n+1}|{\mS}^{n-2}|}\ld |{\mS}^{n-1}|
{\bf 1}_{E}(x+r\sg){\bf 1}_{E}(x-r\sg).$$
Taking integration for both sides of the inequality with respect to the measure element \\
${\rm d}x r^{n-1}{\rm d}r{\rm d}\sg ={\rm d}x{\rm d}y$  gives
$$J\fr{1}{R^n}mes(E\cap B_{R}(v_0))\ge \fr{C_{n,\ld}}{4^{n+1}|{\mS}^{n-2}|}\ld |{\mS}^{n-1}|J
=J\alpha \fr{1}{R^n}mes(B_R(v_0)) $$
where
$$J=\int_{E\cap B_{R/6}(v_0)}{\rm d}x\int_{R/2\le |y|\le 5R/6}
{\bf 1}_{E}(x+y){\bf 1}_{E}(x-y){\rm d}y.$$
Now in order to prove (\ref{5.17}), we need only to prove that $J>0$.
To do this we fix a vector ${\bf e}\in {\bSn}$ and let
\beas&& a=v_0+(2R/3){\bf e},\quad
b=v_0-(2R/3){\bf e},\quad E_a=E\cap B_{R/6}(a),\quad  E_b=
E\cap B_{R/6}(b),\\
&&
I(x)=\int_{{\bRn}}{\bf 1}_{E_a}(x+y)
{\bf 1}_{E_b}(x-y)
{\rm d}y,\quad x\in {\bRn}.\eeas
We have $B_{R/6}(a), B_{R/6}(b)\subset B_R(v_0)\subset K^{\circ}$, and so
$mes(E_a)>0, mes(E_b)>0$, and the function
$x\mapsto I(x)=({\bf 1}_{E_a}*{\bf 1}_{E_b})(2x)$ is continuous on ${\bRn}$
and
$\int_{{\bRn}}I(x){\rm d}x=2^{-n}mes(E_a)mes(E_b)>0$ and so the set
$G=\{x\in {\bRn}\,|\, I(x)>0\}$ is open and non-empty.
 Observe that for any $x,y\in{\bRn}$ satisfying
${\bf 1}_{E_a}(x+y){\bf 1}_{E_b}(x-y)>0$,  i.e. $x+y\in E_a, x-y\in E_b$,
we have by using
$v_0=\fr{1}{2}(a+b)$ and $|a-b|=4R/3$ that $x=\fr{1}{2}(x+y+x-y)\in B_{R/6}(v_0)$
and $R/2\le |y|=\fr{1}{2}|x+y-(x-y)|\le 5R/6$.
This implies that $G\subset B_{R/6}(v_0)\subset K^{\circ}$
and
$I(x)=\int_{R/2\le |y|\le 5R/6}{\bf 1}_{E_a}(x+y)
{\bf 1}_{E_b}(x-y)
{\rm d}y$ and so $mes(E\cap G)>0$ and
$$J\ge \int_{E\cap G}{\rm d}x\int_{R/2\le |y|\le 5R/6}
{\bf 1}_{E_a}(x+y)
{\bf 1}_{E_b}(x-y){\rm d}y=\int_{E\cap G}I(x){\rm d}x>0.$$
Thus (\ref{5.17}) holds true.

{\it Step4.} Finally suppose that $mes(\p K)=0$. Then from
$E\setminus K\subset Z_0$, $K\setminus E\subset (K^{\circ}\setminus E)\cup \p K$, and {\it Step3}
we conclude $mes(E\setminus K)=mes(K\setminus E)=0$.
This completes the proof of the lemma.
\quad $\Box$
\\

\noindent {\bf 4.3.} Now we prove Theorem \ref{theorem4.1} and Theorem \ref{theorem4.2}.
First of all we note that the condition (\ref{4.1}) is invariant under translations and
orthogonal transforms.
\vskip2mm

{\bf Proof of Theorem} \ref{theorem4.1}. Let $K$ satisfy the assumption in
Theorem \ref{theorem4.1}.  The proof consists of five steps.

{\it Step1}. Prove that $K$ is an $n$-dimensional convex body.
We first prove that $({\rm conv}(K))^{\circ}\neq \emptyset$,
where ${\rm conv}(K)$ is the convex hull  of $K$.
Let $a, b\in K$ satisfy
$|a-b|={\rm diam}(K):=h\,(>0)$. Then $a, b\in \p K$. We need to prove that $a_0:=\fr{1}{2}(a+b)\in K$.
To do this
we take $\og\in {\bSn}$ satisfying
$\la \og, a-b\ra=0$ and let
\be x_0=\fr{a+b}{2}+\fr{|a-b|}{2}\og=a_0+\fr{h}{2}\og,\quad
y_0=\fr{a+b}{2}-\fr{|a-b|}{2}\og=a_0-\fr{h}{2}\og.\label{6.CC}\ee
By condition (\ref{4.1}), one of $x_0, y_0$ belongs to $K$.
By replacing $\og$ with $-\og$ we may assume that $x_0\in K$.
Let $t_0=\max\{t\in [1,\infty)\,|\, a_0+t\fr{1}{2}h\og\in K\}$. Then $t_0\ge 1$ and
$x^*_0:=a_0+t_0\fr{1}{2}h\og\in \p K$.
Let $a_0^*\in {\bRn}$ satisfy
$
\fr{1}{2}(a_0+a_0^*)=\fr{1}{2}(a+x_0^*)$, i.e.
$a_0^*=a+x_0^*-a_0=a+t_0\fr{1}{2}h\og.$
Then using $\og\,\bot\,a-b$ we compute $
|a_0-a_0^*|=|a-x_0^*|=\fr{1}{2}h\sqrt{1+t_0^2}>0$,
and
$|a_0^*-b|=h\sqrt{1+\fr{1}{4}t_0^2}>h={\rm diam}(K)$ which implies that $a_0^*\not\in K$.
Now take $\sg=(a_0-a_0^*)/|a_0-a_0^*|\,(=(a_0-a_0^*)/|a-x_0^*|).$
Then
\beas a_0=
\fr{a+x_0^*}{2}+\fr{|a-x_0^*|}{2}\sg,\quad
a_0^*=
\fr{a+x_0^*}{2}-\fr{|a-x_0^*|}{2}\sg.\eeas
Since $a, x_0^*\in \p K$ and $a_0^*\not\in K$, it follows from the condition (\ref{4.1})
that $a_0\in K$.

Replacing $\og$ with ${\bf u}_i$ and using (\ref{4.1}) again we see that for the orthogonal unit vectors ${\bf u}_1=(1,0,...,0), {\bf u}_2=(0,1,0,...,0),..., {\bf u}_n=
(0,...,0,1)$ we have, for every $i\in\{ 1,2,...,n\}$, that
either $a_0+\fr{1}{2}h{\bf u}_i\in K $ or $
a_0-\fr{1}{2}h{\bf u}_i\in K.$ Thus there are $
k_i\in\{0,1\}$ such that
$a_i:=a_0+(-1)^{k_i}\fr{1}{2}h {\bf u}_i\in K,\, i=1,2,...,n$, and so
$a_1-a_0, a_2-a_0, ..., a_n-a_0$ are linearly independent and thus
${\rm conv}(\{a_0, a_1, ..., a_n\})$ is an $n$-dimensional convex body. This implies
$({\rm conv}(K))^{\circ}\supset ({\rm conv}(\{a_0, a_1, ..., a_n\}))^{\circ}\neq \emptyset$.

Next we prove that
\be \p {\rm conv}(K)\subset \p K.\label{6.2}\ee
Let $\Gm$ be the set of extreme points of ${\rm conv}(K)$. It is well-known that
$\Gm\subset K\setminus K^{\circ}=\p K$. Thus to prove (\ref{6.2}) it suffices to prove that
$\p {\rm conv}(K)\subset \Gm.$
Take any $x_0\in \p {\rm conv}(K)$. By separation theorem
and $({\rm conv}(K))^{\circ}\neq\emptyset$,
there exist $\sg\in {\bSn}$ and $y_0\in \p {\rm conv}(K)$ with
$\la y_0,\sg\ra<\la x_0,\sg\ra$
such that
$H^{(-)}_{x_0}(\sg):=\{x\in {\bRn}\,\,|\,\,\la x-x_0,\sg\ra=0\}$ and
$H^{(+)}_{y_0}(\sg):=
\{x\in {\bRn}\,\,|\,\,\la x-y_0,\sg\ra=0\}$ are parallel supporting planes
of $K$ at $x_0,y_0$ respectively and satisfy
 \be\la x-x_0,\sg\ra\le 0,\quad \la x-y_0,\sg\ra\ge 0 \qquad \forall\, x\in {\rm conv}(K).\label{6.3}\ee
Let
$$d=\max\{|x-y|\,\,|\,\, x\in H^{(-)}_{x_0}(\sg)\cap {\rm conv}(K),\,
y\in H^{(+)}_{y_0}(\sg)\cap {\rm conv}(K)\}$$
and take
$x_1\in H^{(-)}_{x_0}(\sg)\cap {\rm conv}(K), y_1\in H^{(+)}_{y_0}(\sg)\cap {\rm conv}(K)$
such that $d=|x_1-y_1|.$
From $x_0\in H^{(-)}_{x_0}(\sg)\cap {\rm conv}(K),
y_0\in H^{(+)}_{y_0}(\sg)\cap {\rm conv}(K)$ we have
$d\ge |x_0-y_0|>0$. Using (\ref{6.3}) it is easily seen that
$x_1, y_1\in \Gm$. In fact, if $0<t<1$ and $a,b\in {\rm conv}(K)$
are such that $x_1=(1-t)a+ tb$, then using the first inequality in
(\ref{6.3}) and $\la x_1-x_0,\sg\ra=0$ we deduce
$\la a-x_0,\sg\ra=\la b-x_0,\sg\ra=0$ and so
$a,b\in H^{(-)}_{x_0}(\sg)\cap {\rm conv}(K)$. This implies that the equalities
$d=|x_1-y_1|=|(1-t)(a-y_1)+t(b-y_1)|=(1-t)|a-y_1|+t|b-y_1|$ hold
and so $|a-y_1|=|b-y_1|=d$,
$\la a-y_1, b-y_1\ra=|a-y_1||b-y_1|$, hence
$(a-y_1)/|a-y_1|=(b-y_1)/|b-y_1|$, and thus $a=b$. This proves that $x_1\in \Gm$.
Using the second inequality in
(\ref{6.3}) one also proves that $y_1\in \Gm$. By the condition (\ref{4.1})
and $x_1, y_1\in \Gm\subset \p K$, we then conclude that
either $\fr{1}{2}(x_1+y_1)+\fr{1}{2}|x_1-y_1|\sg\in K$ or
$\fr{1}{2}(x_1+y_1)-\fr{1}{2}|x_1-y_1|\sg\in K$ and thus using (\ref{6.3})
we deduce that
\beas&& {\rm either}\quad -\fr{\la x_1-y_1, \sg\ra}{2}+\fr{|x_1-y_1|}{2}
=\Big\la \fr{x_1+y_1}{2}+\fr{|x_1-y_1|}{2}\sg-x_0,\,\sg\Big\ra
\le 0,\\
&&{\rm or}\quad \fr{\la x_1-y_1, \sg\ra}{2}-\fr{|x_1-y_1|}{2}=\Big\la \fr{x_1+y_1}{2}-\fr{|x_1-y_1|}{2}\sg-y_0,\,\sg\Big\ra
\ge 0.\eeas
Since $\la x_1-y_1, \sg\ra\le |x_1-y_1|$, this gives the equality
 $|x_1-y_1|=\la x_1-y_1,\sg\ra$.
From this we have
$d=\la x_1-y_1,\sg\ra=\la x_0-y_1,\sg\ra\le |x_0-y_1|\le d$ and so
$|x_0-y_1|=d$ which (as shown above) implies
that $x_0\in\Gm$. This proves $\p {\rm conv}(K)\subset \Gm$ and thus
(\ref{6.2}) holds true.

Next we prove that ${\rm conv}(K)= K$. Since, by convexity,
${\rm conv}(K)=\overline{({\rm conv}(K))^{\circ}}$,
we need only to prove that $({\rm conv}(K))^{\circ}\subset K$.\,
Let $x\in ({\rm conv}(K))^{\circ}$. Take $a\in {\rm conv}(K)$
such that
$|a-x|=\max\{|y-x|\,\, | \,\, y\in {\rm conv}(K)\}$, and then
let $H_x=\{y\in {\bRn}\,\,|\,\, \la y-x,\, a-x\ra=0\}$ and take $b\in H_x\cap  {\rm conv}(K)$ such that
$|b-x|=\max\{|y-x|\,\, | \,\, y\in H_x\cap {\rm conv}(K)\}.$
 From $x\in ({\rm conv}(K))^{\circ}$ it is easily seen that $|a-x|>0, |b-x|>0$. It is
 obvious that $a\in \p{\rm conv}(K)$. We show that it also holds $b\in \p {\rm conv}(K)$.
Otherwise, $b\not\in \p {\rm conv}(K)$, then $b\in
({\rm conv}(K))^{\circ}$ and so there is $\dt>0$ such that
$B_{\dt}(b)\subset  {\rm conv}(K)$.
So $\wt{b}:=b+\dt\fr{b-x}{|b-x|}\in B_{\dt}(b)\subset  {\rm conv}(K)$
and it also holds $\wt{b}\in H_x$, hence $\wt{b}\in H_x\cap {\rm conv}(K)$.
By definition of $|b-x|$ we deduce a contradiction:
$|b-x|+\dt=|\wt{b}-x|\le |b-x|$.  This proves that $b\in \p {\rm conv}(K)$.

Connecting (\ref{6.2}) we conclude that $a,b\in \p K$.
From  $\la b-x,a-x\ra=0$ we have
$|a+b-2x|^2=|a-x|^2+|x-b|^2=|a-b|^2$ and for the unit vector
$\sg=\fr{a+b-2x}{|a-b|}$ we have
$x=\fr{1}{2}(a+b)-\fr{1}{2}|a-b|\sg.$ Let
$y=\fr{1}{2}(a+b)+\fr{1}{2}|a-b|\sg.$ Then
$|y-x|=|a-b|>|a-x|$
and so the maximality of $|a-x|$ implies that $y\not\in {\rm conv}(K)$ hence
$y\not\in K$. Thus, by the condition (\ref{4.1}) and $a, b\in \p K$
we conclude that $x\in K$. This proves that
$({\rm conv}(K))^{\circ}\subset K$ and thus
${\rm conv}(K)= K$.

{\it Step2.} Prove that $K$ has constant width.
We will use a characterization of convex bodies of constant width \cite{CG}.
 Let $H_1, H_2$ be any two
parallel supporting planes of $K$. Then there exist
$\sg\in{\bSn}, p\in H_1\cap \p K$ and
$q\in H_2\cap \p K$ with $\la p,\sg\ra\neq \la q,\sg\ra$ such that
$H_1=\{x\in {\bRn}\,\, |\,\,\la x-p,\sg\ra=0\},$
$H_2=\{x\in {\bRn}\,\, |\,\,\la x-q,\sg\ra=0\}$. We may assume that
$\la p,\sg\ra>\la q,\sg\ra$. Then
\be\la  x-p,\sg\ra\le 0,\quad  \la x-q, \sg\ra\ge 0 \qquad \forall\, x\in K.
\label{pq}\ee
Let us prove that $\sg=(p-q)/|p-q|$. Since $p, q\in \p K,$  by (\ref{4.1}) we may assume that
$\fr{1}{2}(p+q)+\fr{1}{2}|p-q|\sg\in K.$
Then using the first inequality in (\ref{pq}) we deduce
$\fr{1}{2}|p-q|\le \fr{1}{2}\la p-q,\sg\ra$ which implies that $|p-q|= \la p-q,\sg\ra$
(because $\la p-q,\sg\ra\le |p-q|$) hence
$\sg=(p-q)/|p-q|$.
Similarly if
$\fr{1}{2}(p+q)-\fr{1}{2}|p-q|\sg\in K$, then
using the second inequality in (\ref{pq}) we also deduce
$\sg=(p-q)/|p-q|$. This proves that the chord
$[p, q]=\{(1-t)p+t q\,\,|\,\, t\in [0,1]\}$ is orthogonal to
$H_1$ and $H_2$, that is,
$(p-q)/|p-q|$ is a common normal vector of $H_1$ and
$H_2$. According to a characterization of convex bodies of constant width in page 53 of \cite{CG},  $K$ has constant width.

{\it Step3}. Having proven that $K$ is a convex body of constant width,
we now make a further preparation for the rest of two steps. Recalling basic
properties of convex bodies of constant width, it is well-known that

$\bullet$ The insphere and circumsphere of $K$ are concentric, and their radii, $r$ and $R$ respectively, satisfy
\be r+R={\rm diam}(K),\quad R\ge r\ge \big(\sqrt{2(n+1)/n}-1\big)R.\label{6.7}\ee
See \cite{CG}, \cite{MMO2019}.

$\bullet$  For any $p\in \p K$
there exists ${\bf e}\in {\bSn}$ such that
$\la x-p,{\bf e}\ra>0$ for all $ x\in K\setminus \{p\}.$ In particular,
every point in $\p K$ is an extreme point of $K$.

$\bullet$  $\p K$ and ${\mS}^{n-1}$ are homeomorphic.

Since the condition (\ref{4.1}) holds under translations and orthogonal transforms,
we may assume without loss of generality that the common center of the insphere and circumsphere of $K$ is the origin
$0$.
Thus we can write the insphere and circumsphere of $K$ as $B_r:=B_r(0), B_R:=B_R(0)$
respectively and we have $B_r\subset K\subset B_R.$
Then using $r+R={\rm diam}(K)$ it is easily seen that the following implication relations hold:
\be a\in (\p B_r)\cap \p K\,
\Longrightarrow\, -\fr{R}{r}a\in (\p B_R)\cap \p K;\quad
 b\in (\p B_R)\cap \p K\,
\Longrightarrow\, -\fr{r}{R}b\in (\p B_r)\cap \p K\,.\label{6.7*}\ee
Now let $a_0\in (\p B_r)\cap \p K, b_0\in  (\p B_R)\cap \p K$ be such that
\be |a_0-b_0|={\rm dist}\big((\p B_r)\cap \p K, (\p B_R)\cap \p K\big). \label{6.6}\ee
It is easily seen that
\be \la a_0, b_0\ra\ge 0. \label{6.6*}\ee
In fact, applying Lemma \ref{lemma5.1} to the set $K$ and
the unit vector ${\bf n}=b_0/|b_0|$ we have
$(\p B_r)_{\bf n}^+\cap \p K\neq \emptyset$.
Take $x_0\in (\p B_r)_{\bf n}^+\cap \p K$. Then
$\la x_0, {\bf n}\ra\ge 0$, i.e.
$\la x_0, b_0\ra\ge 0$. So
$|a_0-b_0|^2\le |x_0-b_0|^2\le r^2+R^2$ which implies that $\la a_0, b_0\ra\ge 0$.

{\it Step4}. Assuming $n=2$, we
prove that $r=R$, i.e. prove that $K$ is a disk. We will use contradiction argument, so
in the following we assume that $r<R$. We will prove that there exists a rectangular
$a_1a_2a_3a_4$ in $K$ such that
the three vertices $a_1, a_2, a_3$ belong to $\p K$ but the
fourth vertex $a_4$ belongs to $K^{\circ}$. Then we can deduce a contradiction.

To do this we denote
$$u(\theta):=(\cos(\theta),\sin(\theta)),\quad \theta\in[0,2\pi].$$
In the following derivation the equality
$\la u(\theta), u(\vartheta)\ra=\cos(\theta-\vartheta)\,\,(\forall\,\theta,\vartheta
\in [0,2\pi] )$ will be often used.
And we use notations in {\it Step3} and recall that $a_0, b_0$
are given by (\ref{6.6}) and satisfy (\ref{6.6*}).
After making an orthogonal linear transform we can assume that
$$b_0=Ru(\theta_1),\,\,\theta_1=\pi/2;\quad a_0=ru(\theta_2)\,\,\,{\rm for\,\,
some\,\,}\,0\le \theta_2\le\pi/2.$$
Then, since $a_0\in(\p B_r)\cap \p K, b_0\in(\p B_R)\cap \p K$, we know from (\ref{6.7*}) that
\be a_1:=-\fr{r}{R}b_0=-ru(\theta_1)\in (\p B_r)\cap \p K,\quad
a_2:=-\fr{R}{r}a_0=-Ru(\theta_2)\in (\p B_R)\cap \p K.\label{6.8}\ee
Let $0<\theta_0<\pi/2$ be such that $\sin(\theta_0)=\fr{r}{R}$. Using the inequality in (\ref{6.7}) for $n=2$, i.e. $r\ge (\sqrt{3}-1)R$, we have $\sin(\theta_0)\ge \sqrt{3}-1>\sqrt{2}/2$ and so
$\theta_0>\pi/4$.

We first prove that the
angular $\theta_2$ in (\ref{6.8}) satisfies
\be \pi/6\le \theta_2<\theta_0.\label{6.9}\ee
To do this let $c_0=ru(\theta_0).$ Then
$ \la c_0, b_0-c_0\ra=0$. We show that
$c_0\in K^{\circ}$.  Otherwise, $c_0\in \p K$ hence
$c_0\in (\p B_r)\cap \p K$, and it follows from
the implication relation in (\ref{6.7*}) that $-\fr{R}{r}c_0\in \p K$.
Using
$ c_0\,\bot\,b_0-c_0$ we deduce a contradiction:
$(r+R)^2=({\rm diam}(K))^2\ge|b_0- (-\fr{R}{r}c_0)|^2=|b_0-c_0|^2+(r+R)^2
>(r+R)^2.$ Thus we must have $c_0\in K^{\circ}.$ From this we see that
there is $\dt>0$ small enough such that $c_0^*:=c+\dt u(\theta_0)=(r+\dt)u(\theta_0)\in K$.
Next we show that $\theta_2<\theta_0$. If $\theta_2=\theta_0$, then $a_0=c_0\in K^{\circ}$ which contradicts $a_0\in \p K$;
if $\theta_0<\theta_2<\pi/2$, then $a_0\in
({\rm conv}(\{0, b_0, c_0^*\}))^{\circ}\subset K^{\circ}$,
which also contradicts $a_0\in \p K$; if
$\theta_2=\pi/2$, then $a_0=ru(\pi/2)=\fr{R-r}{R+r}a_1+\fr{2r}{R+r}b_0
\in  K^{\circ}$ (because $a_1, b_0\in \p K, a_1\neq b_0, 0<r<R,$ and
every point in $\p K$ is an extreme point of $K$) which also contradicts $a_0\in \p K$.
Thus we must have $\theta_2<\theta_0$. In order to prove $\theta_2\ge \pi/6$,
we need to prove that
\be ru(\theta)\in K^{\circ}\qquad \forall\,\theta\in (\theta_2,\pi-\theta_2).\label{6.10}\ee
In fact for any $\theta\in (\theta_2, \pi-\theta_2)$ we have $\sin(\theta)>\sin(\theta_2)$.
Since $|ru(\theta)-b_0|^2=|ru(\theta)-Ru(\pi/2)|^2
=r^2+R^2-2rR\sin(\theta),  |a_0-b_0|^2=
|ru(\theta_2)-Ru(\pi/2)|^2=r^2+R^2-2rR\sin(\theta_2)$, this implies that
$|ru(\theta)-b_0|<|a_0-b_0|$. By the minimality of
$|a_0-b_0|$ we conclude that $ru(\theta)\not\in \p K$ and so
$ru(\theta)\in K^{\circ}.$ This proves (\ref{6.10}).

Now take a sequence $\{\vartheta_k\}_{k=1}^{\infty}\subset (\theta_2, \theta_0)$
satisfying $\vartheta_k\to \theta_2\,\,(k\to\infty)$.
For every $k\in{\mN}$, let ${\bf n}_{k}=u(\vartheta_k+\pi/2)$.
By Lemma \ref{lemma5.1},
$(\p B_r)_{{\bf n}_{k}}^+\cap \p K$ is non-empty.
Take $ru(\vartheta_k^*)\in (\p B_r)_{{\bf n}_{k}}^+\cap \p K$ with
$\vartheta_k^*\in [0,2\pi)$. Then
$\la ru(\vartheta_k^*), {\bf n}_{k}\ra\ge 0$\, i.e.
$\la u(\vartheta_k^*), u(\vartheta_k+\pi/2)\ra=
\sin(\vartheta_k^*-\vartheta_k)\ge 0$. Since $\vartheta_k^*\in [0,2\pi)$ and
$0\le \theta_2<\vartheta_k<\pi/2$, this implies that $\vartheta_k^*-\vartheta_k\in [0,\pi]$
and so $\theta_2<\vartheta_k\le \vartheta_k^*\le \pi+\vartheta_k$.
Also from $ru(\vartheta_k^*)\in \p K$ and
(\ref{6.10}) we deduce that $\vartheta_k^*\not\in(\theta_2,\pi-\theta_2)$ and thus
$\pi-\theta_2\le \vartheta_k^*\le \pi+\vartheta_k.$ Now choose a convergent
subsequence $\{\vartheta_{k_j}^*\}_{j=1}^{\infty}$ of $\{\vartheta_k^*\}_{k=1}^{\infty}$
and let
$\vartheta^*_{\infty}=\lim\limits_{j\to\infty}\vartheta_{k_j}^*$. Then,
from $\vartheta_k\to \theta_2\,\,(k\to\infty)$, we have  that
$\pi-\theta_2\le \vartheta_{\infty}^*\le \pi+\theta_2$ and
$ru(\vartheta^*_{\infty})=\lim\limits_{j\to\infty}ru(\vartheta_{k_j}^*)\in (\p B_r)\cap \p K$.
By (\ref{6.7*}) we conclude that $b^*:=-Ru(\vartheta^*_{\infty})\in (\p B_R)\cap \p K$, and thus
by the minimality of $|a_0-b_0|$ we have
$|a_0-b_0|\le |a_0-b^*|$.
Since $|a_0-b_0|^2=
 |ru(\theta_2)-Ru(\pi/2)|^2= r^2+R^2-2rR\cos(\theta_2-\pi/2),
 |a_0-b^*|^2=
 |ru(\theta_2)+Ru(\vartheta^*_{\infty})|^2
=r^2+R^2+2rR\cos(\theta_2-\vartheta^*_{\infty})$, this implies that
$\cos(\pi/2+\theta_2)=-\cos(\theta_2-\pi/2)\le \cos(\theta_2-\vartheta^*_{\infty})=\cos(\vartheta^*_{\infty}-\theta_2).$
Since
$0<\pi-2\theta_2\le \vartheta^*_{\infty}-\theta_2\le\pi, 0<\pi/2+\theta_2
<\pi$ and since $\cos(\cdot)$ is strictly decreasing on $[0,\pi]$, it follows
that $\pi/2+\theta_2\ge \vartheta^*_{\infty}-\theta_2\ge \pi-2\theta_2$. Thus
$\theta_2\ge \pi/6.$

As mentioned in {\it Step3}, $\p K$ and ${\mS}^1$ are homeomorphic
and the mapping $\vp: {\mS}^1\to \p K$,
$\vp(u)=\rho(u)u, u\in {\mS}^1$, is a homeomorphism,
where
$u\mapsto \rho(u)=\max\{ \rho\in[r, R]\,\,|\,\,\rho u\in \p K\}$
is continuous on ${\mS}^1$. Thus using $u(\theta)=(\cos(\theta),\sin(\theta))$, the boundary
$\p K$ can be written
$$\p K=\{\rho(\theta)u(\theta)\,\,|\,\,\theta\in [0,2\pi]\},\qquad
\rho(\theta):=\rho(u(\theta)).$$
Note that $\theta\mapsto \rho(\theta)$ is continuous on $[0,2\pi]$ and
$r\le \rho(\theta)\le R$ for all $\theta\in [0,2\pi].$

Now we prove that
\be {\rm there\,\,\, exists}\,\,\, a_3\in \p K\,\,\, {\rm such\,\,\,that}\,\, \,\la a_3
-a_2,a_1-a_2\ra =0\label{6.11}\ee
\be {\rm and}\,\,\, a_4\in K^{\circ}\,\,\, {\rm where}\,\,\,\fr{a_2+a_4}{2}=\fr{a_1+a_3}{2}\,\,\,\,
{\rm i.e.}\,\,\, a_4=a_1+a_3-a_2.\label{6.11*}\ee
To prove (\ref{6.11}) we use a continuous function
\beas&&
f(\theta):=\la \rho(\theta)u(\theta)-a_2,a_1-a_2
\ra=\big\la \rho(\theta)u(\theta)+Ru(\theta_2),
-ru(\pi/2)+Ru(\theta_2)\big\ra\\
&&=R^2+\rho(\theta)R\cos(\theta-\theta_2)-r\rho(\theta)\sin(\theta)-rR\sin(\theta_2),\quad \theta\in[\pi/2, \pi].\eeas
Using $0<r<R$, $r\le \rho(\cdot)\le R$, and $\cos(\theta_2)>0$ we have
\beas&& f(\pi/2)
=R^2-r\rho(\pi/2)+\big(\rho(\pi/2)-r\big)R\sin(\theta_2)\ge R^2-r\rho(\pi/2)>0,\\
&& f(\pi)\le R^2 g(\theta_2),\quad {\rm where}\,\,\, g(\theta):=1-\fr{r}{R}\big(\cos(\theta)+\sin(\theta)\big).\eeas
We need to prove that $g(\theta_2)\le 0$. To do this we observe that
$\theta\mapsto g(\theta)$ is convex on $ [\pi/6, \theta_0]$
so that it holds $g(\theta)\le \max\{g(\pi/6), g(\theta_0)\}$ for all $\theta\in [\pi/6, \theta_0]$.
Using $r\ge (\sqrt{3}-1)R$, $\sin(\theta_0)=\fr{r}{R}$, and $\pi/4<\theta_0<\pi/2$ we have
\beas&& g(\pi/6)=1-\fr{r}{R}\big(\fr{\sqrt{3}}{2}+\fr{1}{2}\big)
=1-\fr{r}{R}\fr{1}{\sqrt{3}-1}\le 0,\\
&&
g(\theta_0)
=\cos(\theta_0)\big(\cos(\theta_0)-\sin(\theta_0)\big)<0.\eeas
This proves that
$g(\theta)\le 0$ for $\theta\in [\pi/6,\theta_0]$ and therefore $g(\theta_2)\le 0$ hence $f(\pi)\le 0$
because $\theta_2\in[\pi/6,\theta_0)$. By intermediate value theorem,
there exists $\theta_3\in (\pi/2, \pi]$ such that $f(\theta_3)=0.$
Thus the point
$$a_3:=\rho(\theta_3)u(\theta_3)\in \p K\,\,\,{\rm satisfies}\,\,\, (\ref{6.11}).$$
Next let $a_4$ be defined in (\ref{6.11*}) i.e.
$a_4=a_1+a_3-a_2$. Using (\ref{6.11}) we have
\be
|a_2|^2+|a_4|^2=|a_1|^2+|a_3|^2,\quad
|a_2-a_4|^2=|a_1-a_3|^2.\label{6.12}\ee
From the first equality in (\ref{6.12}) we have
$|a_4|^2=|a_1|^2+|a_3|^2-|a_2|^2
=r^2+\rho(\theta_3)^2-R^2\le r^2$ and so $a_4\in B_r\subset K$.
We then come to prove that
\be a_4\in K^{\circ}.\label{6.13}\ee
Suppose to the contrary that $a_4\not\in K^{\circ}$. Then
$a_4\in(\p B_r)\cap \p K$, $|a_3|=\rho(\theta_3)=R$, and thus
$a_3=Ru(\theta_3)\in\p K$ (with $\pi/2<\theta_3\le \pi$).
Write $a_4=ru(\theta_4)$ for some $\theta_4\in[0,2\pi)$. From
$a_4-a_3=a_1-a_2$ we have $|a_4-a_3|^2=|a_1-a_2|^2$.
Since $
|a_4-a_3|^2
=|ru(\theta_4)-Ru(\theta_3)\big)|^2=r^2+R^2-2rR\cos(\theta_4-\theta_3),\
|a_1-a_2|^2=|-ru(\pi/2)+Ru(\theta_2)|^2
=r^2+R^2-2rR\sin(\theta_2)$, this gives
\be\cos(\theta_4-\theta_3)=\sin(\theta_2)\quad (>0).\label{6.14}\ee
On the other hand from $a_4-a_1=
a_3-a_2$ and (\ref{6.11}) we have
$0= \la a_3-a_2,a_1-a_2\ra=\la a_4-a_1
,a_1-a_2\ra= rR\cos(\theta_4)\cos(\theta_2)-r\big(1+
\sin(\theta_4)\big)\big(r-R\sin(\theta_2)\big)$
which gives
$$rR\cos(\theta_4)\cos(\theta_2)
=r(1+\sin(\theta_4))\big(r-R\sin(\theta_2)\big)\ge 0$$
where we have used $R\sin(\theta_2)<R\sin(\theta_0)=r.$
Since $\pi/6\le \theta_2<\pi/2$ so that $rR\cos(\theta_2)>0$, this implies that
$\cos(\theta_4)\ge 0$ and thus
$\theta_4\in [0,\pi/2]\cup[3\pi/2, 2\pi).$
Next from $a_4\in(\p B_r)\cap \p K$ we have
$|a_4-b_0|\ge |a_0-b_0|.$
Since
$|a_4-b_0|^2=|ru(\theta_4)-Ru(\pi/2)|^2
=r^2+R^2-2rR\sin(\theta_4),
|a_0-b_0|^2=r^2+R^2-2rR\sin(\theta_2),$ this implies that
$\sin(\theta_4)\le \sin(\theta_2).$  If
$\theta_4\in[0,\pi/2]$, then
$0\le \pi/2-\theta_4<\theta_3-\theta_4\le\pi$, and so from (\ref{6.14})
we have $\cos(\theta_3-\theta_4)=\sin(\theta_2)=\cos(\pi/2-\theta_2)$
hence
$\theta_3-\theta_4=\pi/2-\theta_2$ and so $\pi/2\ge \theta_4=\theta_3-\pi/2+\theta_2>\theta_2>0$ hence $\sin(\theta_4)>\sin(\theta_2)$
(here we used the strict inequality $\theta_3>\pi/2$). This
contradicts $\sin(\theta_4)\le \sin(\theta_2)$. If
$\theta_4\in[3\pi/2, 2\pi)$, then
$\pi/2\le \theta_4-\theta_3<3\pi/2$
so that we still get a contradiction:
$0<\sin(\theta_2)=\cos(\theta_3-\theta_4)=\cos(\theta_4-\theta_3)\le 0.$
Thus it must hold $a_4\in  K^{\circ}$.

Note that from $a_1\in \p K, |a_1|= r, a_4\in K^{\circ}, |a_4|\le r$ and $|a_2|=R>r$, we have
$a_4\neq a_1$, $a_2\neq a_1$, $a_2\neq a_4$, hence $a_2-a_3=a_1-a_4\neq 0, |a_1-a_3|=|a_2-a_4|>0$.
Recalling the facts mentioned in {\it Step3}, for the boundary points $a_1, a_3\in \p K$,
there are $\og_1, \og_3\in {\mS}^1$ such that
\be \la x-a_1,\og_1\ra> 0\quad \forall\, x\in K\setminus\{a_1\};\quad
\la x-a_3,\og_3\ra>0\quad \forall\, x\in K\setminus \{a_3\}.\label{6.SP}\ee
By $a_4\in K^{\circ}$ we have
$t^*:=\max\{t\in [1,\infty)\,|\, a_2+t(a_4-a_2)\in K\}>1$.
Let
$$a_4^*=a_2+t^*(a_4-a_2),\quad \og=\fr{a_1-a_3}{|a_1-a_3|}.$$
Then $a_4^*\in\p K, \og\in {\mS}^1$, and so by the condition (\ref{4.1}) we have
\be {\rm either}\quad u:=\fr{a_2+a_4^*}{2}+\fr{|a_2-a_4^*|}{2}\og\in K
\quad{\rm or}\quad
v:=\fr{a_2+a_4^*}{2}-\fr{|a_2-a_4^*|}{2}\og\in K.\label{6.UV}\ee
On the other hand using $|a_1-a_3|=|a_2-a_4|$ we compute
\beas&&\fr{a_2+a_4^*}{2}=a_2+\fr{t^*}{2}(a_4-a_2),\quad \fr{|a_2-a_4^*|}{2}\og=\fr{t^*}{2}(a_1-a_3)\\
&&\Longrightarrow\quad
u-a_1
=(1-t^*)(a_2-a_1),\quad v-a_3=(1-t^*)(a_2-a_3).\eeas
Since $a_2\in K\setminus\{a_1, a_3\}$, using (\ref{6.SP}) and $t^*>1$ we deduce
$$\la u-a_1,\og_1\ra=(1-t^*)\la a_2-a_1,\og_1\ra<0,\quad
\la v-a_3,\og_3\ra=(1-t^*)\la a_2-a_3,\og_3\ra<0 $$
and so using (\ref{6.SP}) again we conclude that $u\not\in K$ and $v\not\in K$. But this
contradicts (\ref{6.UV}). This contradiction proves that $r=R$, i.e. $K$ is a disk.

{\it Step5}. We now make induction argument on the dimension $n\ge 2$. In {\it Step4}
we have proved that the theorem holds for $n=2$. Suppose the theorem holds
for a dimension $n-1\ge 2$. Let $K\subset {\bRn}$ be a compact set having at least two elements and satisfying (\ref{4.1}). In the previous steps we have proved that
$K$ is a convex body of constant width. With the notations and properties shown in {\it Step3} we may assume that the common center of the insphere and circumsphere of $K$ is the origin
$0$ and we denote by $B_r^n=B_r^n(0),
B_R^n=B_R^n(0)$ the insphere and circumsphere of $K$. Then
$0<r\le R$, $B_r^n\subset K\subset B_R^n$, and $r+R={\rm diam}(K)$.

We now prove that $r=R$.
Let $a_0\in (\p B_r^{n})\cap \p K, b_0\in  (\p B_R^{n})\cap \p K$
satisfy (\ref{6.6}).
By (\ref{6.7*}) we have
$-\fr{r}{R}b_0\in (\p B_r^{n})\cap \p K$. Since $n\ge 3$, there is ${\bf e}_0\in {\bSn}$ such that
${\bf e}_0\,\bot\,\{a_0, b_0\}$. Let
$\Pi=\{v\in {\mR}^{n}\,\,|\,\, \la v, {\bf e}_0\ra=0\}$ and let
$\{{\bf e}_1, {\bf e}_2, ...,{\bf e}_{n-1}\}\subset \Pi$ be an orthnormal base of
$\Pi$. Then $\{{\bf e}_0, {\bf e}_1, {\bf e}_2, ...,{\bf e}_{n-1}\}$ is an orthnormal base of
${\bRn}$. It is easily seen that the linear mapping ${\cal A}: {\mR}^{n-1}\to\Pi $ defined by
${\cal A}(x)=\sum_{i=1}^{n-1}x_i{\bf e}_i,\, x=(x_1, x_2, ...,x_{n-1})\in {\mR}^{n-1}$,
is an isometry:
$\la {\cal A}(x),{\cal A}(y)
\ra=\la x,y\ra,\, |{\cal A}(x)|=|x|,\, x, y\in {\mR}^{n-1}$,
and ${\mR}^{n-1}={\cal A}^{-1}(\Pi).$ Let $K_1={\cal A}^{-1}(K\cap \Pi).$ Then $K_1\subset {\mR}^{n-1}$ is a compact
infinite set. We now prove that $K_1$ satisfies the condition (\ref{4.1}) for the
$n-1$-dimensional case. To do this we need to prove that
\be {\cal A}^{-1}\big((\p K)\cap \Pi\big)= \p K_1.\label{6.16}\ee
Take any $x_0={\cal A}^{-1}(v_0)\in{\cal A}^{-1}((\p K)\cap \Pi)$
with $v_0\in (\p K)\cap \Pi.$ Since $v_0\in \p K$, by {\it Step3}
there is an ${\bf e}\in {\mS}^{n-1}$ such that
\be\la v-v_0, {\bf e}\ra>0\qquad \forall\, v\in K\setminus\{v_0\}.
\label{6.17}\ee
Write
${\bf e}=\sum_{i=0}^{n-1}t_i{\bf e}_i=:t_0{\bf e}_0+\wt{{\bf e}}$ with
$\wt{{\bf e}}=\sum_{i=1}^{n-1}t_i{\bf e}_i\in \Pi.$
It is easily seen that $\wt{{\bf e}}\neq 0.$ In fact from $B^n_r(0)\subset K$,
(\ref{6.17}), and $v_0 \bot {\bf e}_0$ we have $-\dt{\bf e}\in K\setminus\{v_0\}$ for some
$0<\dt\le r$ and so
$0>-\dt=\la -\dt{\bf e},{\bf e}\ra> \la v_0, {\bf e}\ra=
\la v_0, \wt{{\bf e}}\ra$. Therefore $\wt{{\bf e}}\neq 0$.
From $\wt{{\bf e}}/|\wt{{\bf e}}|\in\Pi$ we have
${\bf n}:={\cal A}^{-1}(\wt{{\bf e}}/|\wt{{\bf e}}|)\in
{\mS}^{n-2}$. For any $x\in K_1\setminus\{x_0\}$,
let $v={\cal A}(x)$ and note that $v_0={\cal A}(x_0)$.
We have $v\in (K\cap \Pi)\setminus\{v_0\}$ and it follows from (\ref{6.17})
and $v,v_0\in \Pi$ that
$$\la x-x_0, {\bf n}\ra
=\la {\cal A}^{-1}(v-v_0), {\cal A}^{-1}((\wt{{\bf e}}/|\wt{{\bf e}}|)
\ra
=\la v-v_0, \wt{{\bf e}}/|\wt{{\bf e}}|
\ra=\fr{1}{|\wt{{\bf e}}|}\la v-v_0, {\bf e}\ra>0.$$
Since $x_0\in K_1$ and the inequality $\la x-x_0, {\bf n}\ra>0$ holds for all $x\in K_1\setminus \{x_0\}$, it follows that $x_0\in \p K_1$. This proves that ${\cal A}^{-1}\big((\p K)\cap \Pi\big)\subset\p K_1$.
Next for any $x_0\in \p K_1$, since ${\mR}^{n-1}={\cal A}^{-1}(\Pi)$,
letting $v_0={\cal A}(x_0)$, then $v_0\in \Pi$ and $x_0={\cal A}^{-1}(v_0)$.
By definition of boundary point, for any $\vep>0$
we have $(B_{\vep}^{n-1}(x_0))^{\circ}\cap  K_1\neq\emptyset,
(B_{\vep}^{n-1}(x_0))^{\circ}\cap  K_1^c\neq\emptyset$. Let
$x_1\in (B_{\vep}^{n-1}(x_0))^{\circ}\cap  K_1, x_2\in (B_{\vep}^{n-1}(x_0))^{\circ}\cap  K_1^c$.
Then, since $K_1={\cal A}^{-1}\big(K\cap \Pi\big)$ and $K_1^c=
{\cal A}^{-1}\big(K^c\cap \Pi\big)$, we can write
$x_1={\cal A}^{-1}(v_1), x_2={\cal A}^{-1}(v_2)$ for some $v_1\in K\cap \Pi, v_2\in K^c\cap \Pi$.
Then from
$|v_1-v_0|={\cal A}(x_1-x_0)|=|x_1-x_0|<\vep,
|v_2-v_0|={\cal A}(x_2-x_0)|=|x_2-x_0|<\vep$ we see that
$(B_{\vep}^n(v_0))^{\circ}\cap K\neq \emptyset,
(B_{\vep}^n(v_0))^{\circ}\cap K^c\neq \emptyset$
and so $v_0\in \p K$ and thus $v_0\in (\p K)\cap \Pi$.
Thus $x_0={\cal A}^{-1}(v_0)\in {\cal A}^{-1}((\p K)\cap \Pi)$.
This proves that $\p K_1\subset {\cal A}^{-1}((\p K)\cap \Pi)$ and thus (\ref{6.16}) holds true.

To see that $K_1$ satisfies the condition (\ref{4.1}) for the
$n-1$-dimensional case, we take any  $x, y\in \p K_1={\cal A}^{-1}\big((\p K)\cap \Pi\big), $
and any $\sg\in {\mS}^{n-2}$. Let $v={\cal A}(x), v_*
={\cal A}(y),
\og={\cal A}(\sg)$. Then $v, v_*\in (\p K)\cap \Pi$ and from
$|\og|=|{\cal A}(\sg)|=|\sg|=1$ we have $\og\in {\mS}^{n-1}\cap \Pi$.
Since $K$ satisfies the condition (\ref{4.1}) for dimension $n$, this implies that
either $\fr{1}{2}(v+v_*)+\fr{1}{2}|v-v_*|
\og\in K$ or $\fr{1}{2}(v+v_*)-\fr{1}{2}|v-v_*|
\og\in K$. Since all
$v, v_*,
\og$ belong to $\Pi$, this also implies that either
$\fr{1}{2}(v+v_*)+\fr{1}{2}|v-v_*|\og\in K\cap \Pi,$ or
$\fr{1}{2}(v+v_*)-\fr{1}{2}|v-v_*|
\og\in K\cap \Pi$. Thus using
$|x-y|=|{\cal A}^{-1}(v-v_*)|=|v-v_*|$
we deduce
$$\fr{x+y}{2}\pm \fr{|x-y|}{2}
\sg
={\cal A}^{-1}\Big(\fr{v+v_*}{2}\pm\fr{|v-v_*|}{2}
\og\Big)$$
and so
either $ \fr{1}{2}(x+y)+\fr{1}{2}|x-y|
\sg\in {\cal A}^{-1}(K\cap \Pi)=K_1$, or $\fr{1}{2}(x+y)-
\fr{1}{2}|x-y|
\sg\in {\cal A}^{-1}(K\cap \Pi)=K_1$. This proves that
 $K_1$ satisfies the condition (\ref{4.1}) for the $n-1$-dimensional case.

By induction hypotheses, $K_1\subset {\mR}^{n-1}$ is a
ball. Write $K_1=B^{n-1}_{r_0}(c_0)$.
From ${\bf e}_0\,\bot\,\{a_0, b_0\}$ we have
$a_0, b_0, -\fr{r}{R}b_0\in (\p K)\cap \Pi$ and thus from (\ref{6.16})
we see that
$$x_0:={\cal A}^{-1}(a_0),\,\, p_0:={\cal A}^{-1}(b_0),\,\,q_0:={\cal A}^{-1}(-\fr{r}{R}b_0)
\,\,\, {\rm all\,\,beong\,\, to}\,\,\p K_1.$$  Since ${\cal A}$ is an isometry,
we have  ${\rm diam}(K_1)
={\rm diam}(K\cap\Pi)
\le {\rm diam}(K).$
On the other hand we have
${\rm diam}(K_1)\ge |p_0-q_0|
=(1+\fr{r}{R})|{\cal A}^{-1}(b_0)|=R+r=
{\rm diam}(K).$
Thus  ${\rm diam}(K_1)=|p_0-q_0|=R+r$ and so
$$K_1=B^{n-1}_{r_0}(c_0)\quad
{\rm with}\quad r_0=\fr{R+r}{2},\quad c_0=\fr{p_0+q_0}{2}.$$
By $q_0=-\fr{r}{R}p_0$ we have
$c_0
=\fr{1}{2}(1-\fr{r}{R})p_0=\fr{1}{2}(1-\fr{r}{R}){\cal A}^{-1}(b_0).$
Since $x_0={\cal A}^{-1}(a_0)\in \p K_1=\p B_{r_0}^{n-1}(c_0)$
and since $\la a_0,b_0\ra\ge 0$,  it follows that
$(\fr{R+r}{2})^2=r_0^2=|x_0-c_0|^2
=|a_0-\fr{1}{2}(1-\fr{r}{R})b_0|^2
\le |a_0|^2+|\fr{1}{2}(1-\fr{r}{R})b_0|^2=r^2+(\fr{R-r}{2})^2$ and so
$Rr=(\fr{R+r}{2})^2-(\fr{R-r}{2})^2\le r^2.$
Since $0<r\le R$, this implies $r=R$ and so
$K$ is an $n$-dimensional ball. This completes the proof of Theorem
\ref{theorem4.1}.\qquad $\Box$
\vskip1mm
{\bf Remark}. Theorem \ref{theorem4.1}
for the case $n=2$ may be also proven by using
the Mizel's Conjecture (which has been already proven, see e.g. Theorem 4.2.1 in \cite{MMO2019}), that is,
by proving that {\it every rectangle with three vertices on the
boundary $\p K$ also has its fourth vertex on $\p K$.}  The proof
for ``the fourth vertex $\not\in K^{\circ}$ " is the same as in the
above proof in {\it Step4} (using contradiction argument),
but the  proof for ``the fourth vertex $\in K$" is difficult.
\vskip2mm

{\bf Proof of Theorem} \ref{theorem4.2}. Let $\wt{E}\subset E$ be a Borel subset such that $Z:=E\setminus\wt{E}$ is a null set. Then using ${\bf 1}_{\wt{E}}(v)={\bf 1}_{E}(v)-{\bf 1}_{Z}(v)$ we have
$${\bf 1}_{\wt{E}}(v){\bf 1}_{\wt{E}}(v_*)(1-{\bf 1}_{\wt{E}}(v'))(1-
{\bf 1}_{\wt{E}}(v_*'))
\le {\bf 1}_{E}(v){\bf 1}_{E}(v_*)(1-{\bf 1}_{E}(v'))
(1-{\bf 1}_{E}(v_*'))+2({\bf 1}_{Z}(v')
+ {\bf 1}_{Z}(v_*')\big).$$
Since $mes(Z)=0$, this implies
$$\inttt_{{\bRRSn}}\big({\bf 1}_{Z}(v')+ {\bf 1}_{Z}(v_*')\big){\rm d}v{\rm d}v_* {\rm d}\sg \\
=2|{\mS}^{n-1}|\intt_{{\bRRn}}{\bf 1}_{Z}(v) {\rm d}v {\rm d}v_*=0$$
and so $\wt{E}$ also satisfies the condition (\ref{4.2}).
Thus after a modification on a null set we can assume that $E$ is a Borel set. Now the condition (\ref{4.2}) implies that
for almost every $(v,v_*)\in E\times E$ we have
$(1-{\bf 1}_{E}(v'))(1-{\bf 1}_{E}(v_*'))=0$ a.e. $\sg\in {\bSn}$, which implies that
$$|{\mS}^{n-1}_{v,v_*}(E)|=|{\bSn}|\qquad {\rm a.e.}\quad (v,v_*)\in E\times E.$$
By Lemma \ref{lemma5.5} (with $\ld =1$), there is a null set $Z_0\subset E$
such that $K=\overline{E\setminus Z_0}$ is compact and the equality
$|{\mS}^{n-1}_{v,v_*}(K)|=|{\bSn}|$ holds for all $v,v_*\in K$.
Since $K$ is compact implies that ${\mS}^{n-1}_{v,v_*}(K)$ is compact,
it follows that ${\mS}^{n-1}_{v,v_*}(K)={\bSn}$ for all $v,v_*\in K$.
This shows that $K$ satisfies the condition (\ref{4.1}) and thus
by Theorem \ref{theorem4.1}, $K=B_R(v_0)$ is an $n$-ball.
This also implies that $mes(\p K)=0$ and thus by Lemma \ref{lemma5.5} we conclude
 $mes\big((K\setminus E)\cup(E\setminus K)\big)=0$.
$\quad \Box$
\\


\begin{thebibliography}{99}

\bibitem{A} Arkeryd, L.: On the Boltzmann equation,
$Arch.\ Rat. \ Mech. \ Anal.$ {\bf 45}: 1-34 (1972).

\bibitem{weak-coupling} Benedetto, D.; Pulvirenti, M.; Castella, F.; Esposito, R.: On
the weak-coupling limit for bosons and fermions. Math. Models Methods
Appl. Sci. {\bf 15}: 1811-1843  (2005).

\bibitem{CIP} Cercignani, C.; Illner, R.;  Pulvirenti, M.: {\it
The  Mathematical
Theory  of  Dilute  Gases } (Springer, New York, 1994).


\bibitem{CG}  Chakerian, G. D.;  Groemer, H.: Convex bodies of constant width. Convexity and its applications, 49-96, Birkh\"{a}user, Basel, 1983.

\bibitem{CC} Chapman, S.; Cowling, T.G.: {\it The  Mathematical Theory  of
Non-Uniform  Gases } (Third Edition, Cambridge University Press, 1970).

\bibitem{ESY} Erd\"{o}s, L.; Salmhofer, M.; Yau, H.-T.: On the
quantum Boltzmann equation. J. Stat. Phys. {\bf 116}: 367-380  (2004).


\bibitem{Lu2001} Lu, X. : On spatially homogeneous solutions of a modified Boltzmann equation for Fermi-Dirac particles. J. Statist. Phys.{\bf 105} , no. 1-2, 353-388  (2001).


\bibitem{LS} Lukkarinen, J.; Spohn, H.:
Not to normal order--notes on the kinetic limit for weakly interacting quantum fluids.
 J. Stat. Phys. {\bf 134}: 1133-1172  (2009).



\bibitem{MMO2019} Martini, Horst; Montejano, Luis; Oliveros, D\'{e}borah: Bodies of constant width. An introduction to convex geometry with applications. Birkh\"{a}user/Springer, Cham, 2019.


\bibitem{M}  Montejano, Luis: A characterization of the Euclidean ball in terms of concurrent sections of constant width. Geom. Dedicata {\bf 37}, no. 3, 307-316  (1991).


\bibitem {Nordheim} Nordheim, L.W.: On the kinetic methods in the new statistics
and its applications in the electron theory of conductivity.
Proc. Roy. Soc. London Ser. A  {\bf 119}, 689-698  (1928).

\bibitem{TM} Truesdell, C.; Muncaster, R.G.: {\it  Fundamentals
Maxwell's  Kinetic
Theory  of  a  Simple  Monoatomic  Gas}
(Academic Press, New York, 1980).

\bibitem {Uehling and Uhlenbeck}Uehling, E.A.; Uhlenbeck, G.E.:
Transport phenomena in Einstein-Bose and Fermi-Dirac gases, I, Phys. Rev.
{\bf 43}: 552-561  (1933).



\end{thebibliography}
\end{document}